\documentclass{amsart}

\makeatletter
\def\subsection{\@startsection{subsection}{2}%
  \z@{.5\linespacing\@plus.7\linespacing}{.5\linespacing}%
  {\normalfont\bfseries}}
\makeatother

\makeatletter
\def\@defaultbiblabelstyle#1{[#1]}
\makeatother

\usepackage{upgreek}
\usepackage{amssymb}
\usepackage{dsfont}
\usepackage{extarrows}
\usepackage{hyperref}
\usepackage{upref}
\usepackage[all]{xy}
\usepackage{tikz}
\usetikzlibrary{calc}

\newtheorem{theorem}{Theorem}[section]
\newtheorem{lemma}[theorem]{Lemma}
\newtheorem{proposition}[theorem]{Proposition}
\newtheorem{corollary}[theorem]{Corollary}

\theoremstyle{definition}
\newtheorem{definition}[theorem]{Definition}
\newtheorem{example}[theorem]{Example}

\theoremstyle{remark}
\newtheorem{remark}[theorem]{Remark}

\numberwithin{equation}{section}
\allowdisplaybreaks[4]

\newcommand\fe{\mathfrak{e}}
\newcommand\Id{\mathrm{Id}}
\newcommand\II{\mathds{1}}
\newcommand\Vg{V_{\text{geom}}}
\newcommand\abs[1]{\lvert #1 \rvert}
\newcommand\dle {\underset{\mathrm{g}}{\leqslant}}
\newcommand\dsim{\underset{\mathrm{g}}{\sim}}
\newcommand\papp{\underset{\mathrm{V}}{\approx}}
\newcommand\psim{\underset{\mathrm{V}}{\sim}}
\DeclareMathOperator{\rank}{rank}

\begin{document}

\title[Reflection representations and homology]{Reflection representations of Coxeter groups and homology of Coxeter graphs}

\author{Hongsheng Hu}
\address{Beijing International Center for Mathematical Research, Peking University, No.\ 5 Yiheyuan Road, Haidian District, Beijing 100871, China}
\email{huhongsheng(at)amss(dot)ac(dot)cn}
\urladdr{\href{https://huhongsheng.github.io/}{https://huhongsheng.github.io}}




\date{November 23, 2023}


\begin{abstract}
We study and classify a class of representations (called generalized geometric representations) of a Coxeter group of finite rank. These representations can be viewed as a natural generalization of the geometric representation. The classification is achieved by using characters of the integral homology group of certain graphs closely related to the Coxeter graph.
On this basis, we also provide an explicit description of those representations on which the defining generators of the Coxeter group act by reflections.
\end{abstract}

\maketitle

\setcounter{tocdepth}{2}
\tableofcontents

\section{Introduction} \label{sec-intro}

Since the seminal paper \cite{KL79} of Kazhdan and Lusztig appeared in 1979, Coxeter groups and their representations have received much more attention than before and huge progresses have been made. But there are still a lot of problems to be investigated.
Compared with the well-established representation theory of (affine) Weyl groups, which are important examples of Coxeter groups, representations for general Coxeter groups are far from being understood. 

Given that the concept of Coxeter groups originates from the pioneering work \cite{Coxeter34, Coxeter35} by H. S. M. Coxeter on Euclidean reflection groups, it is natural to consider representations on which the Coxeter group ``acts like a reflection group''. 
In this paper we are mainly concerned with a class of such representations, called generalized geometric representations, of a Coxeter group of finite rank. 
We classify these representations (see Theorem \ref{thm-IR} and Theorem \ref{thm-IR-gen}) by using characters of the integral homology group of certain graphs closely related to the Coxeter graph.
Moreover, the characters also tell us which generalized geometric representations admit a nonzero bilinear form which is invariant under the group action (see Theorem \ref{prop-bil}).

It is not hard to find the composition factors of a generalized geometric representation (see Subsection \ref{subsec-red} and Subsection \ref{subsec-red-gen}). 
In this way we are able to get a class of irreducible representations of a Coxeter group of finite rank.
These irreducible representations are attached to the second-highest two sided cell(s) in the sense of Kazhdan and Lusztig \cite{KL79, Lusztig83-intrep} (see Remark \ref{rmk-a=1} and \cite{Hu23} for more details).

In the closing section of this paper, the generalized geometric representations and their quotients serve as a foundation for the investigation of the so-called ``reflection representations''.

Below we are presenting the definition of the main subject of our study.

Let $W = \langle s \in S \mid (st)^{m_{st}} = e, \forall s, t \in S \rangle$ be a Coxeter group with finite generator set $S$ and Coxeter matrix $(m_{st})_{s,t \in S}$.
That is, 
\[\text{$m_{ss} =1$, $\forall s\in S$; $m_{st} = m_{ts} \in \mathbb{N}_{\ge 2} \cup  \{\infty\}$, $\forall s,t \in S$ with $s \ne t$.}\]
For convenience, we work over the complex number field $\mathbb{C}$.

\begin{definition} \label{def-refl-rep} \leavevmode
  \begin{enumerate}
    \item Suppose $V$ is a vector space, and $s: V \to V$ is a linear map. If there exists a nonzero vector $\alpha_s \in V$ and a subspace $H_s \subseteq V$ of codimension one such that $s|_{H_s} = \Id_{H_s}$ (the identity map of $H_s$) and $s \cdot \alpha_s = - \alpha_s$, then $s$ is called a \emph{reflection} on $V$.
        In particular, $V = H_s \oplus \mathbb{C} \alpha_s$, and $\alpha_s$ is unique up to a scalar.
        We call $\alpha_s$ a \emph{reflection vector} of $s$, and $H_s$ the \emph{reflection hyperplane} of $s$.
    \item Let $V$ be a representation of $W$. If $s$ acts on $V$ by a reflection for any $s \in S$, then $V$ is called a \emph{reflection representation} of $W$.
    \item Let $V$ be a reflection representation of $W$. 
         If moreover the reflection vectors $\{\alpha_s \mid s \in S\}$ form a basis of $V$, then $V$ is called a \emph{generalized geometric representation} of $W$.
  \end{enumerate}
\end{definition}

Our notion of generalized geometric representations comes from the ``standard'' reflection representation $\Vg$ (called the geometric representation in the classic \cite{Bourbaki-Lie456}).
As a vector space, $\Vg := \bigoplus_{s \in S} \mathbb{C} \alpha_s$ has a basis $\{\alpha_s \mid s\in S\}$, with a bilinear form $(- | -)$ given by
\begin{equation}   \label{eq-B-Vg}
  (\alpha_s|\alpha_t) := - \cos \frac{\uppi}{m_{st}}, \quad \forall s,t \in S.
\end{equation}
Here we regard $\cos \frac{\uppi}{\infty} = 1$.
The action of $W$ is defined by
\begin{equation}   \label{eq-Vg-action}
  s \cdot \alpha_t := \alpha_t - \frac{2 (\alpha_t|\alpha_s)}{(\alpha_s| \alpha_s)} \alpha_s = \alpha_t + 2 \cos \frac{\uppi}{m_{st}} \alpha_s, \quad \forall s,t \in S.
\end{equation}
In other words, the action of $s$ is the orthogonal reflection with respect to the bilinear form $(-|-)$ with reflection vector $\alpha_s$. 
In particular, the bilinear form $(- | -)$ is $W$-invariant in the following sense,
\begin{equation}   \label{eq-Vg-W-inv}
  (w \cdot x | w \cdot y) = (x|y), \quad \forall w\in W, \quad \forall x,y \in \Vg.
\end{equation}
As observed in the literature, the representation $\Vg$ captures many significant properties of Coxeter groups.
See \cite{Bourbaki-Lie456, Humphreys90} for more details.

Some generalizations of the geometric representation have been explored in previous research.
See, for example, \cite{BB05, BCNS15, Donnelly11, Hee91, Krammer09, Vinberg71} and references therein.
These works emphasize various combinatorial aspects including root systems, the numbers game, the fundamental domain for the group action, etc.
These topics occupy a prominent position within the research of Coxeter groups.

\subsection{Outline of the paper}

This paper is organized as follows. 
In Section \ref{sec-intro}, we give our definition of generalized geometric representations and reflection representations. 
In Section \ref{sec-pre}, we recollect some preliminary knowledge about dihedral groups and graphs.

For simplicity, we assume that $m_{rt} < \infty$ for any $r,t \in S$ in Section \ref{sec-IR} and Section \ref{sec-property}.
In Section \ref{sec-IR}, we examine the structure of generalized geometric representations and provide a classification of their isomorphism classes. 
Under our assumption, a generalized geometric representation of $W$ can be uniquely determined by a datum $((k_{rt})_{r,t \in S, r \ne t}, \chi)$.
Here $(k_{rt})_{r,t \in S, r \ne t}$ is a specific set of natural numbers, and $\chi$ is a character (i.e., a one-dimensional representation) of the first integral homology group $H_1(\widetilde{G})$ of a certain graph $\widetilde{G}$.
Our approach is sketched as follows.
Consider arbitrary two elements $r,t \in S$ generating a dihedral subgroup of $W$.
The corresponding reflection vectors $\alpha_r, \alpha_t$ span a subrepresentation of this dihedral subgroup, which corresponds to a number $k_{rt}$ with $1 \le k_{rt} \le \frac{m_{rt}}{2}$.
We then analyze how these subrepresentations of various dihedral subgroups ``glue'' together to form the whole generalized geometric representation.
It turns out that the information of ``gluing'' is encoded in $\chi$.

Section \ref{sec-property} consists of four topics concerning some properties of the generalized geometric representations.
Our first consideration is their reducibility.
In this respect, generalized geometric representations behave like the geometric representation $\Vg$ (see Subsection \ref{subsec-red}).
After that, we investigate a larger class of reflection representations, called R-representations, which are closely related to the generalized geometric representations (see Subsection \ref{subsec-R}).
Next, Subsection \ref{subsec-bil} gives an answer to the interesting question of which generalized geometric representations admit nonzero $W$-invariant bilinear forms.
Finally, in Subsection \ref{subsec-dual} we examine the structure of the dual of an arbitrary generalized geometric representation, which is also a reflection representation.

In Section \ref{sec-gen}, we drop the assumption $m_{rt} < \infty$ and present relevant results.

In Section \ref{sec-refl-rep}, we consider arbitrary reflection representations without any linearly independence condition on the reflection vectors.
Roughly speaking, such a representation factors through an R-representation studied in Subsection \ref{subsec-R} of a quotient of the Coxeter group $W$.
At this point, we obtain an explicit description of any reflection representation of $W$.

\subsection*{Acknowledgements}

The author is deeply grateful to Professor Nanhua Xi for his patient guidance and insightful discussions.
The author would also like to thank Tao Gui for valuable exchanges.

\section{Preliminaries} \label{sec-pre}

In the classification of the generalized geometric representations, we will use some graphs and their homology groups to investigate the way of ``gluing'' representations of some dihedral groups.
For convenience, we recall in this section some preliminary knowledge about the representations of dihedral groups and integral homology of graphs.

The base field in this article is assumed to be $\mathbb{C}$ unless otherwise specified.
We use $e$ to denote the unity of a group.
For a representation $\rho: W \to GL(V)$ of a group $W$ and for elements $w \in W$, $v \in V$, we write $w \cdot v$ to indicate the vector $\rho(w)\cdot v$ if there is no ambiguity.
We simply say that $V$ is a representation of $W$ if the action $\rho$ is clear in context.

\subsection{Representations of finite dihedral groups} \label{subsec-f-dih}
Let $D_m := \langle r,t \mid r^2 = t^2 = (rt)^m = e \rangle$ be a finite dihedral group ($m \in \mathbb{N}_{\ge 2}$), and $V_{r,t} := \mathbb{C}\beta_r \oplus \mathbb{C}\beta_t$ be a vector space with formal basis $\{\beta_r, \beta_t\}$.
For natural numbers $k$ satisfying $1 \leq k \leq \frac{m}{2}$, we define the action $\rho_k : D_m \to GL(V_{r,t})$ by
\begin{equation}   \label{eq-Dm-rho-k}
  \begin{aligned}
    r \cdot \beta_r & := - \beta_r, \quad  & r \cdot \beta_t & := \beta_t + 2\cos\frac{k\uppi}{m} \beta_r,\\
    t\cdot \beta_t & := - \beta_t,  & t \cdot \beta_r & := \beta_r + 2\cos\frac{k\uppi}{m} \beta_t.
  \end{aligned}
\end{equation}
Intuitively, $r,t$ acts on the (real) plane by two reflections with respect to two lines with an angle of $\frac{k \uppi}{m}$, see Figure \ref{rho-k}.

\begin{figure}[ht]
    \centering
    \begin{tikzpicture}
      \draw[dashed] (-1.3,0)--(2,0);
      \draw[dashed] (240:1.3)--(60:2);
      \draw[->] (0,0)--(0,1.5);
      \draw[->] (0,0)--(330:1.5);
      \draw (0.3,0) arc (0:60:0.3);
      \draw[<->] ($(1.9,0) + (300:0.2)$) arc (300:420:0.2);
      \draw[<->] ($(60:1.9) + (0.2,0)$) arc (0:120:0.2);
      \node[left] (ar) at (0,1.5) {$\beta_r$};
      \node[below] (at) at (330:1.5) {$\beta_t$};
      \node[right] (r) at (2.1,0) {$r$};
      \node[right] (t) at (63:2.3) {$t$};
      \node[right] (angle) at (0.2,0.3) {$\frac{k\uppi}{m}$};
    \end{tikzpicture}
    \caption{$\rho_k : D_m \to GL(V_{r,t})$}\label{rho-k}
\end{figure}

If $k < \frac{m}{2}$, then $\rho_k$ is irreducible.
If $m$ is even and $k = \frac{m}{2}$, then $\rho_\frac{m}{2} \simeq \varepsilon_r \oplus \varepsilon_t$ splits into a direct sum of two representations of dimension $1$, where
\begin{equation}  \label{eq-epsilon-rt}
    \varepsilon_r: r \mapsto -1, t \mapsto 1; \quad \quad
    \varepsilon_t: r \mapsto 1, t \mapsto -1.
\end{equation}
We denote by $\II$ the trivial representation, and by $\varepsilon$ the sign representation, i.e.,
\begin{equation} \label{eq-epsilon}
  \II: r,t \mapsto 1; \quad \quad \varepsilon: r,t \mapsto -1.
\end{equation}

\begin{lemma}[{\cite[\S 5.3]{Serre77}}] \label{lem-Dm-rep} 
The following are all the irreducible representations of $D_m$,
\begin{gather*}
    \{\II, \varepsilon\} \cup \{\rho_1, \dots, \rho_{\frac{m-1}{2}}\}, \quad \text{if $m$ is odd}; \\
    \{\II, \varepsilon, \varepsilon_r, \varepsilon_t\} \cup \{\rho_1, \dots, \rho_{\frac{m}{2}-1}\}, \quad \text{if $m$ is even}.
\end{gather*}
These representations are non-isomorphic to each other.
\end{lemma}

The following lemmas are easy to verify.

\begin{lemma} \label{lem-2.4} \leavevmode
  \begin{enumerate}
    \item \label{lem-2.4-1} The $+1$-eigenspace of either $r$ or $t$ in the representation space $V_{r,t}$ of $\rho_k$ $(1 \le k \le \frac{m}{2})$ is one dimensional.
    \item \label{lem-2.4-2} No nonzero vector in $V_{r,t}$ can be fixed by $r$ and $t$ simultaneously.
  \end{enumerate}
\end{lemma}

\begin{lemma}[Schur's lemma]\label{lem-schur} 
For $k < \frac{m}{2}$, an endomorphism of $\rho_k$ must be a $\mathbb{C}$-scalar.
\end{lemma}

\begin{lemma}\label{lem-2.6}
If $d>1$ is a common divisor of $k$ and $m$ $(1 \le k \le \frac{m}{2})$, then $\rho_k$ factors through $D_{\frac{m}{d}}$ as follows:
\begin{equation*}
    \xymatrix{D_m \ar[rr]^{\rho_k} \ar@{->>}[rd]& & GL(V_{r,t}) \\
             & D_{\frac{m}{d}} \ar[ur]_{\rho_{\frac{k}{d}}} &}
\end{equation*}
Here $D_m \twoheadrightarrow D_{\frac{m}{d}}$ is the natural homomorphism.
\end{lemma}

\begin{lemma} \label{lem-2.7}
If $k < \frac{m}{2}$, then there exists a unique up to a $\mathbb{C}^\times$-scalar nonzero bilinear form $(- |-)$ on $V_{r,t}$ which is invariant under the action $\rho_k$ of $D_m$:
\begin{equation}  \label{eq-Bform}
    (\beta_r|\beta_r) = (\beta_t| \beta_t) = 1, \quad (\beta_r|\beta_t) = (\beta_t| \beta_r) = -\cos \frac{k \uppi}{m}.
\end{equation}
In particular, $(- |-)$ is symmetric.
\end{lemma}

\begin{proof}
  Under the ordered basis $(\beta_r, \beta_t)$ of $V_{r,t}$, we denote the representing matrix of $(- | -)$ by $B \in M_2(\mathbb{C})$, and the matrices of $r$ and $t$ by $R$ and $T$ respectively,
  \begin{equation*}
    R = \begin{pmatrix}
          -1 & 2\cos \frac{k \uppi}{m} \\
          0 & 1
        \end{pmatrix}, \quad
    T = \begin{pmatrix}
          1 & 0 \\
          2 \cos \frac{k \uppi}{m} & -1
        \end{pmatrix}.
  \end{equation*}
  Note that $k < \frac{m}{2}$ implies $\cos \frac{k \uppi}{m} \ne 0$.
  The condition that $(-|-)$ is $D_m$-invariant is equivalent to the equations
  \begin{equation}   \label{eq-bil-mat}
    R^\mathrm{T} B R = B, \quad T^\mathrm{T} B T = B.
  \end{equation}
  Here the superscript $(-)^\mathrm{T}$ denotes the transpose of a matrix.
  By direct computations, Equation \eqref{eq-bil-mat} suggests that $B$ must be a $\mathbb{C}^\times$-multiple of
  \begin{equation*}
    \begin{pmatrix}
      1 & -\cos \frac{k \uppi}{m} \\
      -\cos \frac{k \uppi}{m} & 1
    \end{pmatrix}.
  \end{equation*}
  The lemma is proved.
\end{proof}

Recall that a sesquilinear form $[- | -]$ on a complex vector space $V$ is a binary function $[- | -] :  V \times V \to \mathbb{C}$ which is linear in the first variable and conjugate-linear in the second variable.
By the same proof of Lemma \ref{lem-2.7}, we have:

\begin{lemma} \label{lem-fdih-sesqui}
  Suppose $k < \frac{m}{2}$, and $[- | -]$ is a nonzero sesquilinear form on $V_{r,t}$ which is $D_m$-invariant under the action $\rho_k$.
  Then there exists $c \in \mathbb{C}^\times$ such that
  \begin{equation*}
        [\beta_r|\beta_r] = [\beta_t| \beta_t] = c, \quad [\beta_r|\beta_t] = [\beta_t| \beta_r] = -c \cos \frac{k \uppi}{m}.
  \end{equation*}
  Furthermore, $[- | -]$ is Hermitian if and only if $c \in \mathbb{R}^\times$.
\end{lemma}

\subsection{Representations of the infinite dihedral group} \label{subsec-inf-dih}
This subsection can be skipped before the reader gets into Section \ref{sec-gen}.

Let $D_\infty := \langle r,t \mid r^2 = t^2 = e \rangle$ be the infinite dihedral group.
There are four $D_\infty$-representations of dimension 1, namely, $\II, \varepsilon, \varepsilon_r$ and $\varepsilon_t$ (defined in the same way as in Equations \eqref{eq-epsilon-rt} and \eqref{eq-epsilon}).

For $x, y \in \mathbb{C}$, let $\varrho_{x,y}: D_\infty \to GL(V_{r,t})$ be the representation of $D_\infty$ on the space $V_{r,t} = \mathbb{C}\beta_r \oplus \mathbb{C}\beta_t$ defined by
\begin{align*}
    r \cdot \beta_r & := - \beta_r, & r \cdot \beta_t & := \beta_t + x \beta_r,\\
    t\cdot \beta_t & := - \beta_t, & t \cdot \beta_r & := \beta_r + y \beta_t.
\end{align*}
Here we use the variant symbol $\varrho$ to distinguish it from the notation $\rho$ in Subsection \ref{subsec-f-dih}.

\begin{lemma} \label{lem-2.12-infdih} \leavevmode
  \begin{enumerate}
    \item \label{lem-2.12-1} Suppose $x, y, x^\prime, y^\prime \in \mathbb{C}^\times$. Then, $\varrho_{x,y} \simeq \varrho_{x^\prime, y^\prime}$ if and only if $x y = x^\prime y^\prime$.
    \item \label{lem-2.12-2} Suppose $x, y \in \mathbb{C}^\times$. Then, $\varrho_{x,y}$ is irreducible if and only if $x y \neq 4$.
    \item \label{lem-2.12-3} Suppose $x, y \in \mathbb{C}^\times$. Then $\varrho_{x,0} \simeq \varrho_{1,0}$ and  $\varrho_{0,y} \simeq \varrho_{0,1}$.
        Moreover, we have $\varrho_{0,0} \simeq  \varepsilon_r \oplus \varepsilon_t$.
  \end{enumerate}
\end{lemma}

\begin{proof}
  Denote
  \begin{equation*}
    R_{x,y} := \begin{pmatrix}
                 -1 & x \\
                 0 & 1
              \end{pmatrix}, \quad
    T_{x,y} := \begin{pmatrix}
                 1 & 0 \\
                 y & -1
              \end{pmatrix}.
  \end{equation*}
  If $\varrho_{x,y} \simeq \varrho_{x^\prime, y^\prime}$, then there exists an invertible matrix
  $P=(\begin{smallmatrix}
      a & b \\
      c & d
    \end{smallmatrix}) \in GL_2 (\mathbb{C})$
  such that
  \begin{equation}   \label{eq-inf-dih-rep-1}
    P R_{x,y} = R_{x^\prime,y^\prime} P, \quad P T_{x,y} = T_{x^\prime,y^\prime} P.
  \end{equation}
  A direct computation yields $a x = d x^\prime$, $d y = a y^\prime$, $b = c =0$ and hence $a,d \ne 0$.
  Thus $x y = x^\prime y^\prime$.
  Conversely, if $x y = x^\prime y^\prime$, then the invertible matrix $P =
  (\begin{smallmatrix}
     x^\prime & 0 \\
     0 & x
   \end{smallmatrix})$
   satisfies Equation \eqref{eq-inf-dih-rep-1} (note that $x,y,x^\prime,y^\prime \in \mathbb{C}^\times$).
   Thus $\varrho_{x,y} \simeq \varrho_{x^\prime, y^\prime}$.

   For \eqref{lem-2.12-2}, note that the vector $\beta_r + \beta_t \in V_{r,t}$ is fixed by $\varrho_{2,2}(r)$ and $\varrho_{2,2}(t)$ simultaneously.
   Thus when $xy = 4$ the representation $\varrho_{x,y}$ ($\simeq \varrho_{2,2}$ by \eqref{lem-2.12-1}) is reducible.
   Suppose now $xy \ne 4$.
   If $\varrho_{x,y}$ is reducible, then there is a common eigenvector of $\varrho_{x,y}(r)$ and $\varrho_{x,y}(t)$ in $V_{r,t}$.
   The $-1$-eigenvector (up to a scalar) of $\varrho_{x,y}(r)$ is $\beta_r$, but $\beta_r$ is not an eigenvector of $\varrho_{x,y}(t)$ since $y \ne 0$.
   The $+1$-eigenvector of $\varrho_{x,y}(r)$ is $x \beta_r + 2 \beta_t$, but $\varrho_{x,y}(t) (x \beta_r + 2 \beta_t) = x \beta_r + (xy -2) \beta_t$ is not a multiple of $x \beta_r + 2 \beta_t$ since $x \ne 0$ and $xy \ne 4$.
   Thus $\varrho_{x,y}$ is irreducible if $xy \ne 4$.
   The verification of \eqref{lem-2.12-3} is also straightforward.
\end{proof}

To simplify the notations and avoid any ambiguity, for any $z \in \mathbb{C}$, choose and fix $u = u(z) \in \mathbb{C}$ such that $u^2 = z$, and if $z \in \mathbb{R}^+$ then we choose $u$ to be positive.
We then define $\varrho_z := \varrho_{u,u}$.

When $z = 4 \cos^2 \frac{k \uppi}{m}$ for some $k,m \in \mathbb{N}$ such that $1 \leq k < \frac{m}{2}$, $\varrho_z$ factors through $D_m$ as follows:
\begin{equation*}
    \xymatrix{D_\infty \ar[rr]^{\varrho_{4 \cos^2 \frac{k \uppi}{m}}} \ar@{->>}[rd]& & GL(V_{r,t}) \\
             & D_{m} \ar[ur]_{\rho_{k}} &}
\end{equation*}

When $z = 4$, $\varrho_4 = \varrho_{2,2}$ is the geometric representation of $D_\infty$.
In this case, the vector $\beta_r + \beta_t$ spans a subrepresentation isomorphic to $\II$, and the quotient is isomorphic to $\varepsilon$.

When $z = 0$, we have $\varrho_0 = \varrho_{0,0} \simeq \varepsilon_r \oplus \varepsilon_t$.

Besides, denote $\varrho_r^t := \varrho_{1,0}$ and $\varrho_t^r := \varrho_{0,1}$.
Then both $\varrho_r^t$ and $\varrho_t^r$ are indecomposable.
The representation $\varrho_r^t$ has a subrepresentation $\varepsilon_r$ with quotient $\varepsilon_t$, while $\varrho_t^r$ has a subrepresentation $\varepsilon_t$ with quotient $\varepsilon_r$.

In addition, there is an ``exotic'' representation (which we will not use in this article) of $D_\infty$ defined by $\varrho_{\varepsilon}^{\II}: r \mapsto (\begin{smallmatrix}
  -1 & 0 \\
  0 & 1
  \end{smallmatrix}),
  t \mapsto (\begin{smallmatrix}
              -1 & 1 \\
              0 & 1
            \end{smallmatrix})$.
It has $\varepsilon$ as its subrepresentation with quotient $\II$.

\begin{lemma}\label{lem-infdih-ind-rep}
  All the indecomposable representations of $D_\infty$ of dimension $1$ or $2$  are given by:
  \begin{equation}   \label{eq-list-infdih}
    \{\II, \varepsilon, \varepsilon_r, \varepsilon_t\} \cup \{ \varrho_r^t, \varrho_t^r, \varrho_{\varepsilon}^{\II}\} \cup \{\varrho_z \mid z \in \mathbb{C}^\times\}.
  \end{equation}
  These representations are non-isomorphic to each other.
\end{lemma}

\begin{proof}
  Clearly $\II, \varepsilon, \varepsilon_r$ and $\varepsilon_t$ are the only $D_\infty$-representations of dimension one.

  Suppose $\varrho$ is an indecomposable representation of $D_\infty$ of dimension $2$.
  Since $r^2 = e$, the linear map $\varrho(r)$ is semisimple with possible eigenvalues $\pm 1$.
  Note that $\varrho(r)$ is neither
  $(\begin{smallmatrix}
      1 & 0 \\
      0 & 1
    \end{smallmatrix})$ nor
  $(\begin{smallmatrix}
      -1 & 0 \\
      0 & -1
    \end{smallmatrix})$. 
    Otherwise, since $\varrho(t)$ is also semisimple, $\varrho$ would be decomposable.
  Thus, the matrix of $\varrho(r)$ is similar to
  $(\begin{smallmatrix}
      -1 & 0 \\
      0 & 1
    \end{smallmatrix})$.
  The same result holds for $t$.
  Denote by $V$ the representation space of $\varrho$, by $\beta_r$ the unique (up to scalars) $-1$-eigenvector of $\varrho(r)$ in $V$, and by $\beta_t$ the $-1$-eigenvector of $\varrho(t)$.

  If $\beta_r$ and $\beta_t$ are proportional, then we denote by $\gamma_r$ the $+1$-eigenvector of $\varrho(r)$ in $V$.
  Then $(\beta_r, \gamma_r)$ is a basis of $V$.
  Under this ordered basis, $\varrho(r) = (
  \begin{smallmatrix}
    -1 & 0 \\
    0 & 1
  \end{smallmatrix})$, and $\varrho (t) =
  (\begin{smallmatrix}
     -1 & a \\
     0 & 1
   \end{smallmatrix})$ for some $a \in \mathbb{C}^\times$.
  One may easily verify that $\varrho \simeq \varrho_\varepsilon^\II$.

  If $\beta_r$ and $\beta_t$ are linearly independent, then under the ordered basis $(\beta_r, \beta_t)$ we have $\varrho(r) =
  (\begin{smallmatrix}
     -1 & x \\
     0 & 1
   \end{smallmatrix})$ and $\varrho(t) =
   (\begin{smallmatrix}
      1 & 0 \\
      y & -1
    \end{smallmatrix})$ for some $x,y \in \mathbb{C}$.
  Therefore $\varrho \simeq \varrho_{x,y}$.
  If $x = 0$, then $y \ne 0$ and $\varrho \simeq \varrho_t^r$.
  If $y = 0$, then $x \ne 0$ and $\varrho \simeq \varrho_r^t$.
  If $xy \ne 0$, then $\varrho \simeq \varrho_z$ where $z = xy$.

  In view of Lemma \ref{lem-2.12-infdih} and the structure of $\varrho_r^t, \varrho_t^r$ and $\varrho_{\varepsilon}^{\II}$, the representations in the list \eqref{eq-list-infdih} are non-isomorphic to each other.
\end{proof}

\begin{remark}
  In fact, all the irreducible representations of $D_\infty$ are included in the list \eqref{eq-list-infdih}.
  They are: $\{\II, \varepsilon, \varepsilon_r, \varepsilon_t\} \cup \{\varrho_z \mid z \in \mathbb{C}^\times \setminus \{4\}\}$.
\end{remark}

Similar to the proof of Lemma \ref{lem-2.7}, we have the following lemma.

\begin{lemma}\label{lem-2.11}
  Suppose $\varrho \in \{\varrho_r^t, \varrho_t^r\} \cup \{\varrho_z \mid z \in \mathbb{C}^\times\}$.
  \begin{enumerate}
    \item \label{lem-infdih-schur} An endomorphism of $\varrho$ must be a scalar.
    \item \label{2.4.4} There is a unique up to a $\mathbb{C}^\times$-scalar $D_\infty$-invariant nonzero bilinear form $(-|-)$ on the representation space of $\varrho$. The representing matrix of $(-|-)$ (with respect to the ordered basis $(\beta_r, \beta_t)$) is
         \begin{gather*}
           \begin{pmatrix}
            1 & -\frac{u}{2} \\
            -\frac{u}{2} & 1
           \end{pmatrix}, \text{ if } \varrho = \varrho_z, z \neq 0; \\
           \begin{pmatrix}
            1 & 0 \\
            0 & 0
           \end{pmatrix}, \text{ if } \varrho = \varrho_t^r; \\
           \begin{pmatrix}
            0 & 0 \\
            0 & 1
          \end{pmatrix}, \text{ if } \varrho = \varrho_r^t.
         \end{gather*}
  \end{enumerate}
\end{lemma}

\begin{remark}
  Different from Lemma \ref{lem-fdih-sesqui}, there may not be a nonzero $D_\infty$-invariant sesquilinear form on the representation space of $\varrho_z$, for general $z \in \mathbb{C}^\times$.
\end{remark}

\subsection{Graphs and their homologies} \label{subsec-graph}
As mentioned in the introduction, we will ``glue'' together various representations of dihedral groups, and the way of ``gluing'' is encoded in a character of the first integral homology group of some graph.
In this subsection we present some basic notions and results on graphs and their homologies.
One may also refer to \cite[Ch. 4]{Sunada13}.

By definition, an \emph{(undirected) graph} $G = (S,E)$ consists of a set $S$ of vertices and a set $E$ of edges.
Each edge in $E$ is an unordered binary subset $\{s,t\}$ of $S$.

For our purpose, we only consider finite graphs without loops and multiple edges (i.e., $S$ is a finite set; there is no edge of the form $\{s,s\}$; and each pair $\{s,t\}$ occurs at most once in $E$).
Such a graph can be viewed naturally as a finite simplicial complex: vertices as $0$-simplices, edges as $1$-simplices.
In this way, the graph is regarded as a topological space.
Thus we have the notion of \emph{connected components} of a graph.

A sequence $(s_1, s_2, \dots, s_n)$ of vertices is called a \emph{path} in $G$ if $\{s_i, s_{i+1}\} \in E$, $\forall i$.
If $s_1 = s_n$, then we say the path is a \emph{closed path}.
We do not distinguish the closed paths $(s_1, \dots, s_{n-1}, s_1)$ and $(s_2, \dots, s_{n-1}, s_1, s_2)$.
Further, if $s_1, \dots, s_{n-1}$ are pairwise distinct in this closed path, then we call the path a \emph{circuit}.

For a connected graph $G$, a \emph{spanning tree} is defined to be a subgraph $T = (S,E_0)$ with the same vertex set $S$ and an edge set $E_0 \subseteq E$, such that $T$ is connected and there is no circuit in $T$.
If $G$ is not required to be connected, then we can choose arbitrarily a spanning tree for each component of $G$, and the union of these trees is called a \emph{spanning forest} of $G$.
Spanning forests always exist, but they are not unique in general.

Since a graph $G$ is regarded as a simplicial complex, we can consider its homology groups.
In the sense of \cite{Munkres84}, the chain complex (with coefficients in $\mathbb{Z}$) of $G$ looks like:
\begin{equation*}
    0 \to C_1(G) \xrightarrow{\partial} C_0(G) \to 0.
\end{equation*}
Here  $C_0(G)$ is a free abelian group with free generator set $S$, and $C_1(G)$ is a free abelian group generated by all edges endowed with an orientation (an oriented edge is denoted in the form $(s,t)$).
By convention, we regard $(s,t) = - (t,s)$ in $C_1(G)$ for an edge $\{s,t\}$.
The map $\partial$ is defined by $\partial (s,t) = t - s$.

The first homology group $H_1(G) = \ker \partial$ is a finitely generated free abelian group since $C_1(G)$ is such a group (see, for example, \cite[Lemma 11.1]{Munkres84}).
We regard a circuit $(s_1, s_2, \dots, s_n, s_1)$ as the element $(s_1, s_2) + \dots + (s_{n-1}, s_n) + (s_n, s_1)$ in $C_1(G)$.
Clearly, it is also an element in $H_1(G)$.

Fix a spanning forest $G_0 = (S,E_0)$ of $G$, and endow each edge in $E$ with an orientation.
For any oriented edge $\mathfrak{e} = (s_1,s_2) \in E \setminus E_0$, there is a unique circuit in $(S, E_0 \cup \{\mathfrak{e}\})$ of the form $(s_1, s_2, \dots, s_n, s_1)$.
We denote this circuit by $c_\fe$.

\begin{lemma}\label{lem-2.12}
   With the spanning forest $G_0$ fixed as above, we have \[H_1(G) = \bigoplus_{\fe \in E \setminus E_0} \mathbb{Z} c_\fe.\]
\end{lemma}

\begin{proof}
  Clearly the elements $\{c_\fe \mid \fe \in E \setminus E_0\}$ are $\mathbb{Z}$-linearly independent in $H_1(G)$ since all edges in $c_\fe$ lie in $E_0$ except $\fe$.
  Thus we have an injection \[\bigoplus_{\fe \in E \setminus E_0} \mathbb{Z} c_\fe \hookrightarrow H_1(G).\]

  Regard $H_1(G)$ as a subgroup of $C_1(G)$.
  Then elements in $H_1(G)$ can be written in the form $\sum_{\fe \in E} a_\fe \fe$ with $a_\fe \in \mathbb{Z}$.
  Suppose $\sum_{\fe \in E} a_\fe \fe \in H_1(G)$.
  Let
  \begin{equation*}
    d := \sum_{\fe \in E} a_\fe \fe  - \sum_{\fe \in E \setminus E_0} a_\fe c_\fe \in H_1(G).
  \end{equation*}
  Then any edge that occurs in $d$ with a nonzero coefficient belongs to $E_0$, the edge set of the forest $G_0$.
  It is impossible for such an element, $d$, to be a closed cycle unless $d = 0$.
  Therefore, $H_1(G) = \bigoplus_{\fe \in E \setminus E_0} \mathbb{Z} c_\fe$.
\end{proof}

\subsection{The Coxeter graph}

Let $(W,S)$ be a Coxeter group with Coxeter matrix $(m_{rt})_{r,t \in S}$.
The associated Coxeter graph $G$ is a graph with labels on some edges.
The set of vertices is nothing but the set $S$ of generators.
Two vertices $r,t \in S$ are joined by an edge if $m_{rt} \ge 3$.
In this case the edge $\{r,t\}$ is labelled by $m_{rt}$.
By convention, the label $m_{rt}$ may be omitted if $m_{rt} = 3$.
The number $m_{rt}$ is merely a label, rather than a multiplicity of the edge $\{r,t\}$.
We denote again by $H_1(G)$ the first integral homology group of $G$ while forgetting the labels on its edges.

If $G$ is connected, we say that the Coxeter group is \emph{irreducible}.

\section{Classification of the generalized geometric representations} \label{sec-IR}
Let $(W,S)$ be a Coxeter group of finite rank (i.e., $\abs{S} < \infty$) with Coxeter matrix $(m_{rt})_{r,t \in S}$.
For simplicity, from here to the end of Section \ref{sec-property}, we assume:
\begin{equation} \label{eq-fin-condition}
  m_{rt} < \infty, \quad \forall r,t \in S.
\end{equation}
This section is devoted to a classification of the isomorphism classes of generalized geometric representations of $W$.

\subsection{The GGR-datum \texorpdfstring{$(k_{rt}, a_r^t)_{r,t \in S, r\ne t}$}{(k, a)}} \label{subsec-struIR}

We begin with the following lemma.

\begin{lemma} \label{lem-refl}
  Let $V$ be a vector space and $s : V \to V$ be a reflection with a reflection vector $\alpha_s$.
  Then for any $v \in V$ we have $v - s\cdot v \in \mathbb{C} \alpha_s$.
\end{lemma}

\begin{proof}
  Recall that we have $V = H_s \oplus \mathbb{C} \alpha_s$ where $H_s$ is the reflection hyperplane (see Definition \ref{def-refl-rep}).
  For any $v \in V$, write $v = v_s + a \alpha_s$, where $v_s \in H_s$, $a \in \mathbb{C}$.
  Then $s \cdot v = v_s - a \alpha_s = v - 2a \alpha_s$.
  It follows that $v - s\cdot v \in \mathbb{C} \alpha_s$.
\end{proof}

\begin{lemma}\label{lem4.2}
  Let $V = \bigoplus_{s \in S} \mathbb{C} \alpha_s$ be a generalized geometric representation of $W$, where $\alpha_s$ is a reflection vector of $s$.
  \begin{enumerate}
    \item \label{4.2-1} Suppose $r, t \in S$, $r \ne t$. The vectors $\alpha_r, \alpha_t$ span in $V$ a subrepresentation of the dihedral subgroup $\langle r,t \rangle$.
    \item \label{4.2-2} This representation of $\langle r,t \rangle$, say $\phi_{rt}: \langle r,t \rangle \to GL(\mathbb{C} \langle \alpha_r, \alpha_t \rangle)$,  is isomorphic to $\rho_{k_{rt}}$ for some $1 \leq k_{rt} \leq \frac{m_{rt}}{2}$.
  \end{enumerate}
\end{lemma}

\begin{proof}
  The fact that $\mathbb{C} \langle \alpha_r, \alpha_t \rangle$ is $\langle r,t \rangle$-invariant follows from Lemma \ref{lem-refl}.

  For \eqref{4.2-2}, if $\phi_{rt}$ is reducible, then it splits into a direct sum of two representations of dimension one since $\langle r,t \rangle$ is a finite group.
  The sign representation $\varepsilon$ of $\langle r,t \rangle$ can not occur in $\phi_{rt}$ because $\alpha_r$ and $\alpha_t$ are linearly independent.
  In view of Lemma \ref{lem-Dm-rep}, $m_{rt}$ must be even and $\phi_{rt} \simeq \varepsilon_r \oplus \varepsilon_t \simeq \rho_{\frac{m_{rt}}{2}}$.
  If $\phi_{rt}$ is irreducible, then by Lemma \ref{lem-Dm-rep} again we have $\phi_{rt} \simeq \rho_{k_{rt}}$ for some $1 \le k_{rt} < \frac{m_{rt}}{2}$.
\end{proof}

\begin{remark}
  If $m_{rt} = 2$, i.e., $rt = tr$, then $k_{rt}$ must be $1$, and $\phi_{rt} \simeq \varepsilon_r \oplus \varepsilon_t$ splits.
\end{remark}

Let $V$ be as in Lemma \ref{lem4.2}.
Suppose $\phi_{rt} \simeq \rho_{k_{rt}}$ as in Lemma \ref{lem4.2}\eqref{4.2-2} for an arbitrary pair $r,t \in S$.
We use the notations $\beta_r, \beta_t$ to indicate the chosen basis of the representation space $V_{r,t}$ of $\rho_{k_{rt}}$ so that the action is defined as in Equation \eqref{eq-Dm-rho-k}.
Fix an isomorphism $\psi: \rho_{k_{rt}} \xrightarrow{\sim} \phi_{rt}$ (but such an isomorphism is not unique).
Clearly, $\psi(\beta_r)$ is a multiple of $\alpha_r$ and $\psi(\beta_t)$ is a multiple of $\alpha_t$, i.e., there exist $a_{r}^t, a_{t}^{r} \in \mathbb{C}^\times$ such that $\psi(\beta_r) = a_r^t \alpha_r$ and $\psi(\beta_t) = a_t^r \alpha_t$.
Thus, we have
\begin{equation*}
  r \cdot (a_t^r \alpha_t) = (a_t^r \alpha_t) + 2 \cos \frac{k_{rt}\uppi}{m_{rt}} (a_r^t \alpha_r), \quad
  t\cdot (a_r^t \alpha_r) = (a_r^t \alpha_r) + 2 \cos \frac{k_{rt}\uppi}{m_{rt}} (a_t^r \alpha_t).
\end{equation*}
Divided by the coefficients on the left hand side, the equations are equivalent to
\begin{equation*}
    r \cdot \alpha_t = \alpha_t + 2 \frac{a_r^t}{a_t^r} \cos \frac{k_{rt}\uppi}{m_{rt}} \alpha_r, \quad \quad
    t\cdot \alpha_r = \alpha_r + 2 \frac{a_t^r}{a_r^t} \cos \frac{k_{rt}\uppi}{m_{rt}} \alpha_t.
\end{equation*}
Note that $r,t \in S$ $(r\ne t)$ are arbitrary.
Therefore the action of $W$ on the representation $V$ is determined by the numbers $(k_{rt}, a_r^t)_{r,t \in S, r \neq t}$.
For later use, we give a name to the sets of such numbers.

\begin{definition}
  If a set of numbers $(k_{rt}, a_r^t)_{r,t \in S, r \neq t}$ satisfies the following:
  \begin{equation*}
    k_{rt} \in \mathbb{N}, \quad 1 \leq k_{rt} = k_{tr} \leq \frac{m_{rt}}{2}, \quad a_r^t \in \mathbb{C}^\times, \quad \forall r,t \in S, r \ne t,
  \end{equation*}
  then $(k_{rt}, a_r^t)_{r,t \in S, r \neq t}$ is called a \emph{GGR-datum} of $W$ (``GGR'' stands for ``generalized geometric representation'').
\end{definition}

Thus, we have seen that any generalized geometric representation of $W$ can be determined by a GGR-datum.

Conversely, suppose we are given a GGR-datum $(k_{rt}, a_r^t)_{r,t \in S, r \neq t}$.
We want to show that these numbers can give rise to a generalized geometric representation of $W$ in the way above.
Define for each $r \in S$ a linear transformation on the vector space $V := \bigoplus_{s \in S} \mathbb{C}\alpha_s$ by:
\begin{equation}  \label{eq-IR-action}
    r \cdot \alpha_r = - \alpha_r, \quad r \cdot \alpha_t = \alpha_t + 2 \frac{a_r^t}{a_t^r} \cos \frac{k_{rt}\uppi}{m_{rt}} \alpha_r, \quad \forall t \in S, t \neq r.
\end{equation}
Then for any $r,t \in S$ with $r \ne t$, the subspace $\mathbb{C} \langle \alpha_r, \alpha_t \rangle$ spanned by $\alpha_r$ and  $\alpha_t$ forms a representation of the dihedral subgroup $\langle r,t \rangle$ isomorphic to $\rho_{k_{rt}}$.

\begin{lemma} \label{lem-defrep}
  The space $V$ is a generalized geometric representation of $W$ under the action defined by Equation \eqref{eq-IR-action}.
\end{lemma}

\begin{proof}
  First of all, we need to verify that \eqref{eq-IR-action} is a well defined $W$-action on $V$, i.e., $r^2 \cdot \alpha_s = (rt)^{m_{rt}} \cdot \alpha_s = \alpha_s$, $\forall r,t,s \in S$.
  By direct computations, we know that $r^2 \cdot \alpha_s = \alpha_s$ whether $r = s$ or not.

  Now suppose $r \ne t$, we consider the action of $(rt)^{m_{rt}}$.
  On the subspace $\mathbb{C} \langle \alpha_r, \alpha_t \rangle$,  $(rt)^{m_{rt}}$ is the identity map since $\mathbb{C} \langle \alpha_r, \alpha_t \rangle$ forms a representation $\rho_{k_{rt}}$ of $\langle r,t \rangle$.
  For $s \neq r,t$ (if it exists), we write $U := \mathbb{C} \langle \alpha_s, \alpha_r, \alpha_t \rangle$.
  Then $\dim U = 3$.
  By Lemma \ref{lem-2.4}\eqref{lem-2.4-1}, there are nonzero vectors $v_1 \in \mathbb{C} \langle \alpha_r, \alpha_s \rangle$, $v_2 \in \mathbb{C} \langle \alpha_r, \alpha_t \rangle$ such that $r \cdot v_1 = v_1$ and $r \cdot v_2 = v_2$.
  If we write $v_1 = c_1 \alpha_r + d_1 \alpha_s$, $v_2 = c_2 \alpha_r + d_2 \alpha_t$ where $c_1, d_1, c_2, d_2 \in \mathbb{C}$, then $d_1, d_2 \ne 0$.
  Therefore, $\dim \mathbb{C} \langle v_1, v_2 \rangle = 2$.
  Similarly, there are $v_3, v_4 \in U$ such that $\dim \mathbb{C} \langle v_3, v_4 \rangle = 2$ and $t \cdot v_3 = v_3$, $t \cdot v_4 = v_4$.
  Thus, there exists a nonzero vector $v \in \mathbb{C} \langle v_1, v_2 \rangle \cap \mathbb{C} \langle v_3, v_4 \rangle \subseteq U$.
  We have $r \cdot v = t \cdot v = v$.
  By Lemma \ref{lem-2.4}\eqref{lem-2.4-2}, $v \notin \mathbb{C} \langle \alpha_r, \alpha_t \rangle$.
  As a result, $\{v, \alpha_r, \alpha_t \}$ is a basis of $U$.
  It follows that $(rt)^{m_{rt}}$ acts by identity on $U$, and $(rt)^{m_{rt}} \cdot \alpha_s = \alpha_s$ since $\alpha_s \in U$.
  To conclude, Equation \eqref{eq-IR-action} gives rise to a $W$-representation structure on $V$.

  From the arguments above, we see that for any fixed $r \in S$ and any $s \in S \setminus \{r\}$, there is a vector $v_s = c_s \alpha_r + d_s \alpha_s \in V$ ($c_s, d_s \in \mathbb{C}$) such that $r \cdot v_s = v_s$ and $d_s \ne 0$.
  Hence, the vectors $\{v_s |s \in S \setminus \{r\}\}$ are linearly independent.
  They span a hyperplane $H_r$ with $r|_{H_r} = \Id_{H_r}$.
  Thus $r$ acts on $V$ by a reflection with reflection vector $\alpha_r$.
  By definition, $V$ is a generalized geometric representation of $W$.
\end{proof}

To conclude, any GGR-datum $(k_{rt}, a_r^t)_{r,t \in S, r \neq t}$ of $W$ defines a generalized geometric representation via Equation \eqref{eq-IR-action}, and any  generalized geometric representation of $W$ can be defined by some GGR-datum in this way.

\subsection{The associated graph \texorpdfstring{$\widetilde{G}$}{tilde G} and the character \texorpdfstring{$\chi$}{chi}} \label{subsec-tldG}

We need to clarify in what cases two different GGR-data  define isomorphic representations.
Before that, we shall associate a graph $\widetilde{G}$ and a character $\chi$ of $H_1(\widetilde{G})$ with an GGR-datum $(k_{rt}, a_r^t)_{r,t \in S, r \neq t}$ in this subsection.

For $r,t \in S$ with $r \ne t$, define
\begin{equation*}
  \widetilde{m}_{rt} := \frac{m_{rt}}{d_{rt}}, \text{ where } d_{rt} := \gcd (m_{rt}, k_{rt}).
\end{equation*}
Then $\widetilde{m}_{rt} \ge 2$.
Moreover, set $\widetilde{m}_{ss} := 1$, $\forall s \in S$.
Then $(\widetilde{m}_{rt})_{r,t \in S}$ is a Coxeter matrix.
It defines a new Coxeter group $(\widetilde{W}, S)$ with the same generator set $S$.
Denote by $\widetilde{G} = (S, E)$ the Coxeter graph of $(\widetilde{W}, S)$.
Then \[E = \Bigl\{\{r,t\} \Bigm| r,t \in S, k_{rt} \ne \frac{m_{rt}}{2}\Bigr\}.\]
We call $\widetilde{G}$  the \emph{associated graph} of the datum $(k_{rt}, a_r^t)_{r,t \in S, r \neq t}$.
Since $\widetilde{G}$ only depends on the numbers $(k_{rt})_{r,t \in S, r \neq t}$, we also say $\widetilde{G}$ is the associated graph of $(k_{rt})_{r,t \in S, r \neq t}$.

\begin{remark} \leavevmode
  \begin{enumerate}
    \item Forgetting the labels on the edges, the graph $\widetilde{G}$ is a subgraph of the Coxeter graph $G$ of $W$ with the same vertex set $S$.
        An edge $\{r,t\}$ in $G$ is an edge in $\widetilde{G}$ if and only if $k_{rt} \ne \frac{m_{rt}}{2}$.
    \item It may happen that $G$ is connected but $\widetilde{G}$ has several components.
    \item We have a canonical homomorphism $W \twoheadrightarrow \widetilde{W}$.
    \item Let $\rho: W \to GL(V)$ be the generalized geometric representation of $W$ defined by the datum
        $(k_{rt}, a_r^t)_{r,t \in S, r \neq t}$.
        By Lemma \ref{lem-2.6}, the representation $\rho$ naturally factors through $\widetilde{W}$ as follows,
        \begin{equation*}
           \xymatrix{W \ar[rr]^{\rho} \ar@{->>}[rd]& & GL(V) \\
             & \widetilde{W} \ar[ur]_{\widetilde{\rho}} &}
        \end{equation*}
        Moreover, $\widetilde{\rho}: \widetilde{W} \to GL(V)$ is a generalized geometric representation of $\widetilde{W}$ defined by the GGR-datum $(\widetilde{k}_{rt}, a_r^t)_{r,t \in S, r \neq t}$ of $\widetilde{W}$, where $\widetilde{k}_{rt} := \frac{k_{rt}}{d_{rt}}$.
  \end{enumerate}
\end{remark}

The chain group $C_1(\widetilde{G})$ is a free abelian group generated by edges $(r,t)$ (with $k_{rt} \ne \frac{m_{rt}}{2}$) endowed with an orientation (see Subsection \ref{subsec-graph}).
We define a one dimensional representation of $C_1(\widetilde{G})$ by
\begin{equation} \label{eq-4.3}
    (r,t) \mapsto \frac{a_t^r}{a_r^t}.
\end{equation}
We denote its restriction to the subgroup $H_1 (\widetilde{G})$ by $\chi: H_1 (\widetilde{G}) \to \mathbb{C}^\times$, and call $\chi$ the \emph{associated character} of the datum $(k_{rt}, a_r^t)_{r,t \in S, r \neq t}$.

\begin{example} \label{eg-chi}
  Suppose the Coxeter graph $G$ of the Coxeter group $(W,S)$ is as follows,
  
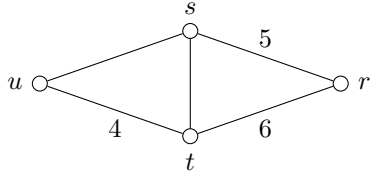
\begin{equation*}
  \begin{tikzpicture}
    \node [circle, draw, inner sep=2pt, label=90:$s$] (s) at (0,0.7) {};
    \node [circle, draw, inner sep=2pt, label=270:$t$] (t) at (0,-0.7) {};
    \node [circle, draw, inner sep=2pt, label=right:$r$] (r) at (2,0) {};
    \node [circle, draw, inner sep=2pt, label=left:$u$] (u) at (-2,0) {};
    \draw (s) -- (t) -- (r) -- (s) -- (u) -- (t);
    \node (tr) at (1, -0.6) {$6$};
    \node (sr) at (1,0.6) {$5$};
    \node (tu) at (-1,-0.6) {$4$};
  \end{tikzpicture}
  \end{equation*}
  Let $k_{su} = k_{st} = 1$ and $k_{sr} = k_{tr} = k_{tu} = 2$.
  Then the associated graph $\widetilde{G}$ is as follows,
  
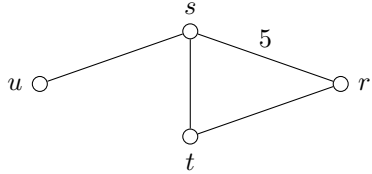
\begin{equation*}
  \begin{tikzpicture}
    \node [circle, draw, inner sep=2pt, label=90:$s$] (s) at (0,0.7) {};
    \node [circle, draw, inner sep=2pt, label=270:$t$] (t) at (0,-0.7) {};
    \node [circle, draw, inner sep=2pt, label=right:$r$] (r) at (2,0) {};
    \node [circle, draw, inner sep=2pt, label=left:$u$] (u) at (-2,0) {};
    \draw (s) -- (t) -- (r) -- (s) -- (u);
    \node (sr) at (1,0.6) {$5$};
  \end{tikzpicture}
  \end{equation*}
  The homology group $H_1(\widetilde{G})$ is isomorphic to $\mathbb{Z}$, generated by the circuit $(s,r,t,s)$.
  Let $a_t^r = x \in \mathbb{C}^\times$ and $a_r^t = a_s^t = a_t^s = a_r^s = a_s^r = 1$.
  Then the associated character $\chi$ takes the value $x$ on the circuit $(s,r,t,s)$.
\end{example}

\subsection{Isomorphism classes}

Suppose we are given two GGR-data $(k_{rt}, a_r^t)_{r,t \in S, r \neq t}$ and $(l_{rt}, b_r^t)_{r,t \in S, r \neq t}$  of $W$.
Denote by $V_1 = \bigoplus_{s \in S} \mathbb{C} \alpha_s$, $V_2 = \bigoplus_{s \in S} \mathbb{C} \alpha_s^\prime$ the two generalized geometric representations of $W$ defined by the two data respectively, with reflection vectors $\{\alpha_s\}_{s \in S}$ and $\{\alpha^\prime_s\}_{s \in S}$.
By Lemma \ref{lem4.2}, we have the following fact.

\begin{lemma} \label{lem-k=l}
  Suppose $\varphi: V_1 \xrightarrow{\sim} V_2$ is an isomorphism of representations.
  Then $k_{rt} = l_{rt}$ for any $r, t \in S$ with $r \ne t$.
\end{lemma}

\begin{proof}
Clearly, $\varphi(\alpha_s)$ is a multiple of $\alpha_s^\prime$ for any $s \in S$.
Thus 
\[\varphi(\mathbb{C} \langle \alpha_r, \alpha_t \rangle) = \mathbb{C} \langle \alpha_r^\prime, \alpha_t^\prime \rangle\] 
for any $r,t \in S$ with $r \ne t$.
The restriction $\varphi|_{\mathbb{C} \langle \alpha_r, \alpha_t \rangle}$ is an isomorphism of $\langle r,t \rangle$-representations between $\mathbb{C} \langle \alpha_r, \alpha_t \rangle$ and $\mathbb{C} \langle \alpha_r^\prime, \alpha_t^\prime \rangle$.
Therefore $\rho_{k_{rt}} = \rho_{l_{rt}}$, and then $k_{rt} = l_{rt}$.
\end{proof}

Suppose now $V_1, V_2$ are isomorphic. 
Then $k_{rt} = l_{rt}$ by Lemma \ref{lem-k=l}.
Thus the two data have the same associated graph $\widetilde{G}$.
Denote by $\chi_1$ and $\chi_2$ the characters of $H_1(\widetilde{G})$ associated with $(k_{rt}, a_r^t)_{r,t \in S, r \neq t}$ and $(k_{rt}, b_r^t)_{r,t \in S, r \neq t}$ respectively (see Subsection \ref{subsec-tldG}).

\begin{lemma}\label{lem-chi1=2}
  Under the assumption $\varphi: V_1 \xrightarrow{\sim} V_2$, it holds that $\chi_1 = \chi_2$.
\end{lemma}

\begin{proof}
Let $(s_1, s_2, \dots, s_n, s_1)$ be a circuit in $\widetilde{G}$.
We write $\alpha_i := \alpha_{s_i}$ and $a_i^j := a_{s_i}^{s_j}$ for convenience, and similar for $m_{ij}, \widetilde{m}_{ij}, k_{ij},$ $\alpha_i^\prime, b_i^j$, etc.

Since $(s_1, s_2)$ is an edge of $\widetilde{G}$, we have $k_{12} \neq \frac{m_{12}}{2}$.
Thus $\mathbb{C} \langle \alpha_1, \alpha_2 \rangle$ and $\mathbb{C} \langle \alpha_1^\prime,  \alpha_2^\prime \rangle$ are irreducible representations of the subgroup $\langle s_1, s_2 \rangle$, both isomorphic to 
\[\rho_{k_{12}}: \langle s_1, s_2 \rangle \to GL(\mathbb{C}\beta_1 \oplus \mathbb{C} \beta_2).\]
We denote the two isomorphisms by $\psi_1, \psi_2$ respectively as follows (not a commutative diagram):
\begin{equation*}
  \xymatrix{ & \mathbb{C} \langle \beta_1, \beta_2 \rangle \ar[rd]_\sim^{\psi_2} \ar[ld]^\sim_{\psi_1} & \\
  \mathbb{C} \langle \alpha_1, \alpha_2 \rangle \ar[rr]^{\varphi|_{\mathbb{C} \langle \alpha_1, \alpha_2 \rangle}}_\sim & & \mathbb{C} \langle \alpha_1^\prime,  \alpha_2^\prime \rangle}
\end{equation*}
As in Subsection \ref{subsec-struIR}, the two maps $\psi_1, \psi_2$ can be chosen so that
\begin{equation}   \label{eq-chi1=2-3}
  \begin{aligned}
     \psi_1(\beta_1) & = a_1^2 \alpha_1, \quad & \psi_1(\beta_2) & = a_2^1 \alpha_2, \\
     \psi_2(\beta_1) & =  b_1^2 \alpha^\prime_1, & \psi_2(\beta_2) & = b_2^1 \alpha^\prime_2.
  \end{aligned}
\end{equation}

Since $\alpha_1, \alpha_1^\prime$ are reflection vectors of $s_1$ in $V_1, V_2$ respectively, there exists  some $x \in \mathbb{C}^\times$ such that $\varphi(\alpha_1) = x \alpha_1^\prime$.
Then
\begin{equation}   \label{eq-chi1=2-4}
  \psi_2^{-1} \varphi \psi_1 (\beta_1) = \psi_2^{-1} \varphi (a_1^2 \alpha_1) = \psi_2^{-1} (x a_1^2 \alpha_1^\prime) = x \frac{a_1^2}{b_1^2} \beta_1.
\end{equation}
Note that $\psi_2^{-1} \varphi \psi_1$ is an automorphism of $\rho_{k_{12}}$.
By Schur's Lemma \ref{lem-schur}, Equation \eqref{eq-chi1=2-4} implies
\begin{equation}   \label{eq-chi1=2-5}
  \psi_2^{-1} \varphi \psi_1 (\beta_2) = x \frac{a_1^2}{b_1^2} \beta_2.
\end{equation}
Therefore, by Equations \eqref{eq-chi1=2-3} and \eqref{eq-chi1=2-5}, we have
\begin{equation*}
  \varphi(\alpha_2) = \frac{1}{a_2^1} \varphi \psi_1(\beta_2) = \frac{1}{a_2^1} \cdot x \frac{a_1^2}{b_1^2} \psi_2(\beta_2) = x \frac{a_1^2 b_2^1}{a_2^1 b_1^2} \alpha_2^\prime.
\end{equation*}

In simple terms, for the edge $(s_1, s_2)$ we have shown that
\begin{equation*}
  \varphi(\alpha_1) = x \alpha_1^\prime \xLongrightarrow{\text{ Schur's lemma }} \varphi(\alpha_2) = x \frac{a_1^2 b_2^1}{a_2^1 b_1^2} \alpha_2^\prime.
\end{equation*}
Do this argument for the edges $(s_2, s_3), \dots, (s_{n-1}, s_n), (s_n, s_1)$ recursively, and finally we get
\begin{equation}   \label{eq-chi1=2-1}
    \varphi(\alpha_1) = x \frac{a_1^2 a_2^3 \cdots a_{n-1}^n a_n^1 b_2^1 b_3^2 \cdots b_n^{n-1} b_1^n}{a_2^1 a_3^2 \cdots a_n^{n-1} a_1^n b_1^2 b_2^3 \cdots b_{n-1}^n b_n^1} \alpha_1^\prime.
\end{equation}
Compare Equation \eqref{eq-chi1=2-1} with $\varphi(\alpha_1) = x \alpha_1^\prime$, then we have
\begin{equation}   \label{eq-chi1=2-2}
    \frac{a_2^1 a_3^2 \cdots a_n^{n-1} a_1^n}{a_1^2 a_2^3 \cdots a_{n-1}^n a_n^1} = \frac{b_2^1 b_3^2 \cdots b_n^{n-1} b_1^n}{b_1^2 b_2^3 \cdots b_{n-1}^n b_n^1}.
\end{equation}
The two sides of Equation \eqref{eq-chi1=2-2} are the images of the arbitrarily chosen circuit $(s_1, s_2, \dots, s_n, s_1)$ under $\chi_1$ and $\chi_2$ respectively.
It follows that $\chi_1 = \chi_2$.
\end{proof}

By Lemma \ref{lem-chi1=2}, once the numbers $(k_{rt})_{r,t \in S, r \ne t}$ are chosen and fixed, there is a well defined map $\Theta = \Theta((k_{rt})_{r,t \in S, r \ne t})$ from the set
\begin{equation*}
  \Biggl\{ \begin{gathered}
            \text{the isomorphism classes of generalized geometric}   \\
             \text{representations defined by } (k_{rt}, a_r^t)_{r,t \in S, r \ne t}
          \end{gathered}
   \Biggm| a_r^t \in \mathbb{C}^\times, \forall r,t \in S \Biggr\}
\end{equation*}
to the set $$\{\text{characters of } H_1(\widetilde{G})\},$$ where $\widetilde{G}$ is the associated graph of $(k_{rt})_{r,t \in S, r \ne t}$.

\begin{lemma} \label{lem-surj}
  For any fixed numbers $(k_{rt})_{r,t \in S, r \ne t}$, the map $\Theta$ is surjective.
\end{lemma}

\begin{proof}
Choose a spanning forest $\widetilde{G}_0 = (S, E_0)$ of the associated graph $\widetilde{G} = (S,E)$.
For any $\mathfrak{e} \in E \setminus E_0$, choose an orientation of $\mathfrak{e}$.
Let $c_\mathfrak{e}$ be defined as in Subsection \ref{subsec-graph}.
Then by Lemma \ref{lem-2.12}, we have
\begin{equation*}
    H_1(\widetilde{G}) = \bigoplus_{\mathfrak{e} \in E \setminus E_0} \mathbb{Z} c_\mathfrak{e}.
\end{equation*}
Therefore, a character $\chi$ of $H_1(\widetilde{G})$ is determined by the values $\{\chi(c_\mathfrak{e}) \mid \mathfrak{e} \in E \setminus E_0 \}$.

Let $\chi$ be an arbitrary character of $H_1(\widetilde{G})$.
For any  $r,t \in S$ with $r \ne t$, let
\begin{equation*}
    a_r^t := \begin{cases}
              \chi(c_{(t,r)}), & \mbox{if } \{t,r\} \in E \setminus E_0, \mbox{ and is oriented as } (t,r); \\
              1, & \mbox{otherwise.} \\

            \end{cases}
\end{equation*}
Then, $\chi$ is the character associated with the datum $(k_{rt}, a_r^t)_{r,t \in S, r \ne t}$ (for an illustrative example, see Example \ref{eg-chi}).
Let $V$ be the generalized geometric representation defined by the datum $(k_{rt}, a_r^t)_{r,t \in S, r \ne t}$, then  $\Theta(V) = \chi$.
\end{proof}

\begin{lemma} \label{lem-inj}
  For any fixed numbers $(k_{rt})_{r,t \in S, r \ne t}$, the map $\Theta$ is injective.
\end{lemma}

\begin{proof}
Suppose $\Theta (V_1) = \Theta (V_2)$, where $V_1 = \bigoplus_{s \in S} \mathbb{C} \alpha_s$ and $V_2 = \bigoplus_{s \in S} \mathbb{C} \alpha_s^\prime$ are generalized geometric representations of $W$ defined by two GGR-data, say, $(k_{rt}, a_r^t)_{r,t\in S, r \ne t}$ and $(k_{rt}, b_r^t)_{r,t\in S, r \ne t}$, respectively.
We need to find an isomorphism $\varphi: V_1 \to V_2$.

It is possible that the associated graph $\widetilde{G}$ have several connected components.
In an arbitrary component of $\widetilde{G}$, we choose and fix a vertex $r$ as a base point, and define
\begin{equation*}
  \varphi(\alpha_r) := \alpha_r^\prime.
\end{equation*}
For any other vertex $s$ in this component, choose a path $(s_1 = r, s_2, \dots, s_n = s)$ connecting $r$ and $s$.
We use the notations $\alpha_i, a_i^j$, etc. as in the proof of Lemma \ref{lem-chi1=2}.
Then we define
\begin{equation}   \label{eq-lem-inj-3}
  \varphi(\alpha_s) = \varphi(\alpha_n) := \frac{a_1^2 a_2^3 \cdots a_{n-1}^n b_2^1 b_3^2 \cdots b_n^{n-1}}{a_2^1 a_3^2 \cdots a_n^{n-1} b_1^2 b_2^3 \cdots b_{n-1}^n} \alpha_n^\prime.
\end{equation}
We shall show that $\varphi(\alpha_n)$ is independent of the choice of the path (but $\varphi$ does depend on the choice of the base point $r$ in each component).
Suppose there is another path connecting $s_1 = r$ and $s_n = s$, say $(s_1, s_p, s_{p-1}, \dots, s_{n+1}, s_n)$, $p \ge n$.
The two paths connected end-to-end form a closed path, i.e.,  $$(s_1, \dots s_n, s_{n+1}, \dots, s_p, s_1).$$
This closed path is an element of $H_1(\widetilde{G})$.
Since $\Theta(V_1) = \Theta(V_2)$, the two characters $\Theta(V_1)$ and $\Theta(V_2)$ takes the same value on this closed path.
In other words, we have the following equality,
\begin{equation}   \label{eq-lem-inj-1}
    \frac{a_2^1 \cdots a_n^{n-1} a_{n+1}^n \cdots a_p^{p-1} a_1^p}{a_1^2 \cdots a_{n-1}^n a_{n}^{n+1} \cdots a_{p-1}^p a_p^1} = \frac{b_2^1 \cdots b_n^{n-1} b_{n+1}^n \cdots b_p^{p-1} b_1^p}{b_1^2 \cdots b_{n-1}^n b_{n}^{n+1} \cdots b_{p-1}^p b_p^1}.
\end{equation}
Equation \eqref{eq-lem-inj-1} is equivalent to
\begin{equation*}
    \frac{a_1^2 a_2^3 \cdots a_{n-1}^n b_2^1 b_3^2 \cdots b_n^{n-1}}{a_2^1 a_3^2 \cdots a_n^{n-1} b_1^2 b_2^3 \cdots b_{n-1}^n} = \frac{a_1^p a_p^{p-1} \cdots a_{n+1}^n b_p^1 b_{p-1}^p \cdots  b_n^{n+1}}{a_p^1 a_{p-1}^p \cdots a_n^{n+1} b_1^p b_p^{p-1} \cdots  b_{n+1}^n}.
\end{equation*}
This indicates that $\varphi(\alpha_n)$ is independent of the choice of the path.
Now, we obtain a linear map $\varphi: V_1 \to V_2$.

Clearly, $\varphi$ is a linear isomorphism.
It remains to verify that $\varphi$ is a homomorphism of representations.
It suffices to check for any $s,t \in S$ that
\begin{equation}   \label{eq-lem-inj-2}
  \varphi(s \cdot \alpha_t) = s \cdot \varphi(\alpha_t).
\end{equation}
If $s = t$, then this is obvious by definition.
If $s$ and $t$ are distinct and not adjacent in $\widetilde{G}$, i.e., $\widetilde{m}_{st} = 2$, then $s \cdot \alpha_t = \alpha_t$ and $s \cdot \alpha_t^\prime = \alpha_t^\prime$, and Equation \eqref{eq-lem-inj-2} is still valid.
Now assume $\widetilde{m}_{st} \geq 3$, then $s$ and $t$ are in the same connected component of $\widetilde{G}$.
Suppose $r$ is the base point chosen in this component, and $(r, s_2, \dots, s)$ is a path connecting $r$ and $s$.
Then $(r, s_2, \dots, s, t)$ is a path connecting $r$ and $t$.
Suppose $\varphi (\alpha_s) = x \alpha_s^\prime$, $x \in \mathbb{C}^\times$.
Then
\begin{equation*}
    \varphi (\alpha_t) = x \frac{a_s^t b_t^s}{a_t^s b_s^t} \alpha_t^\prime.
\end{equation*}
It is then straightforward to verify Equation \eqref{eq-lem-inj-2} by definitions of $V_1$ and $V_2$, as follows,
\begin{align*}
  \varphi(s \cdot \alpha_t) & = \varphi (\alpha_t + 2\frac{a_s^t}{a_t^s} \cos \frac{k_{st}\uppi}{m_{st}} \alpha_s) \\
   & = x \frac{a_s^t b_t^s}{a_t^s b_s^t} \alpha_t^\prime + 2x  \frac{a_s^t}{a_t^s} \cos \frac{k_{st}\uppi}{m_{st}} \alpha_s^\prime;  \\
  s \cdot\varphi(\alpha_t) & = x \frac{a_s^t b_t^s}{a_t^s b_s^t} (s\cdot \alpha_t^\prime) \\
  & = x \frac{a_s^t b_t^s}{a_t^s b_s^t} (\alpha_t^\prime + 2 \frac{b_s^t}{b_t^s} \cos \frac{k_{st}\uppi}{m_{st}} \alpha_s^\prime) \\
  & = x \frac{a_s^t b_t^s}{a_t^s b_s^t} \alpha_t^\prime + 2x \frac{a_s^t}{a_t^s} \cos \frac{k_{st}\uppi}{m_{st}} \alpha_s^\prime  = \varphi(s \cdot \alpha_t).
\end{align*}

Thus, $\varphi: V_1 \to V_2$ is an isomorphism of representations of $W$, and hence $\Theta$ is injective.
\end{proof}

\begin{remark} \label{rmk-322} \leavevmode 
  \begin{enumerate}
    \item \label{rmk-322-1} In fact, by Schur's lemma, Equation \eqref{eq-lem-inj-3} is the only feasible way to define $\varphi(\alpha_s)$ once $\varphi(\alpha_r) = \alpha_r^\prime$ is defined for the base point $r$.
    \item Lemma \ref{lem-inj} is essentially a converse of Lemma \ref{lem-chi1=2}.
  \end{enumerate}
\end{remark}

Now we are ready to establish our main result. 

\begin{theorem}\label{thm-IR}
  The isomorphism classes of generalized geometric representations of $W$ one-to-one correspond to the set of data
  \begin{equation*}
    \left\{ \bigl((k_{rt})_{r, t \in S, r \ne t}, \chi\bigr) \text{ } \middle\vert 
    \begin{gathered}
      \text{ $k_{rt} \in \mathbb{N}$, $1 \leq k_{rt} = k_{tr} \leq \frac{m_{rt}}{2}$, $\forall r, t \in S$, $r \ne t$;} \\
      \text{$\chi$ is a character of  $H_1(\widetilde{G}$), where} \\
      \text{$\widetilde{G}$ is the associated graph of $(k_{rt})_{r,t \in S, r \ne t}$}
    \end{gathered} \right\}.
  \end{equation*}
\end{theorem}

\begin{proof}
  Combine Subsection \ref{subsec-struIR} and Lemmas \ref{lem-k=l}, \ref{lem-chi1=2}, \ref{lem-surj}, \ref{lem-inj}.
\end{proof}

\begin{remark}
  This classification of generalized geometric representations also works if we replace the base field $\mathbb{C}$ by  $\mathbb{R}$.
\end{remark}

\begin{example}
  If every $k_{rt}$ is chosen to be $1$ and $\chi$ is chosen to be the trivial character, then the corresponding representation is nothing but the geometric representation $\Vg$.
\end{example}

\begin{example} \label{eg-affine-A}
  The Coxeter group $(W,S)$ of type $\widetilde{\mathsf{A}}_n$ is defined by $$\langle s_0, s_1, \dots, s_n \mid s_i^2 = (s_i s_{i+1})^3 = e, \forall i = 0, \dots, n \rangle$$ (regard $n+1$ as $0$).
  Its Coxeter graph $G$ is as follows,
  
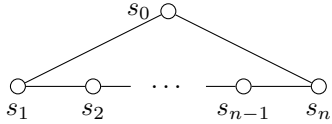
\begin{equation*}
  \begin{tikzpicture}
    \node [circle, draw, inner sep=2pt, label=below:$s_1$] (s1) at (-2,0) {};
    \node [circle, draw, inner sep=2pt, label=below:$s_2$] (s2) at (-1,0) {};
    \node [circle, draw, inner sep=2pt, label=below:$s_{n-1}$] (sn-1) at (1,0) {};
    \node [circle, draw, inner sep=2pt, label=below:$s_n$] (sn) at (2,0) {};
    \node [circle, draw, inner sep=2pt, label=left:$s_0$] (s0) at (0,1) {};
    \draw (sn-1) -- (sn) -- (s0) -- (s1) -- (s2);
    \draw (s2) -- (-0.5,0);
    \draw (sn-1) -- (0.5,0);
    \node (d) at (0,0) {$\dots$};
  \end{tikzpicture}
  \end{equation*}
  We use the notation $(k_{ij}, a_i^j)_{0\le i \ne j \le n}$ to indicate the GGR-datum $(k_{s_is_j}, a_{s_i}^{s_j})_{i \ne j}$, as we did in the proof of Lemma \ref{lem-chi1=2}.
  The condition $1 \leq k_{ij} \leq \frac{m_{ij}}{2}$ forces $k_{ij} = 1$.
  Thus the associated graph $\widetilde{G}$ is the same as $G$.
  The homology group $H_1 (G)$ is isomorphic to $\mathbb{Z}$ generated by the circuit $c = (s_0, s_1, \dots, s_n, s_0)$.
  Giving a character $\chi$ of $H_1 (G)$ is equivalent to giving a number $x \in \mathbb{C}^\times$ and assigning $\chi(c) = x$.
  Such character $\chi$ can be realized by choosing $a_0^n = x$ and $a_n^0 = a_0^1 = a_1^0 =  \dots = 1$.
  Therefore, all generalized geometric representations of $W$ are parameterized by $\mathbb{C}^\times$.
\end{example}

We end this section by a corollary which can be derived from the proof of Lemma \ref{lem-inj} and Remark \ref{rmk-322}\eqref{rmk-322-1}.

\begin{corollary} \label{cor-end-IR}
  Let $V$ be a generalized geometric representation of $W$, and $\widetilde{G}$ be the associated graph.
  Let $g$ be the number of connected components of $\widetilde{G}$.
  Then $\operatorname{End}_W (V) \simeq \mathbb{C}^{\oplus g}$.
  In particular, if $g = 1$, then any endomorphism of $V$ is a scalar multiplication.
\end{corollary}

This recovers the well known fact that any endomorphism of the geometric representation of an irreducible Coxeter group is a scalar (see \cite[Ch. V, Exercise \S 4, (3), (a)]{Bourbaki-Lie456}).

\section{Properties of generalized geometric representations} \label{sec-property}

\subsection{Reducibility of generalized geometric representations} \label{subsec-red}

Recall \cite[Ch. V, \S 4]{Bourbaki-Lie456} that if $(W,S)$ is reducible then the geometric representation $\Vg$ decomposes canonically into a direct sum, with summands corresponding to the components of the Coxeter graph $G$.
If $(W,S)$ is irreducible, then the reducibility of $\Vg$ can be described by the following proposition.

\begin{proposition} [{\cite[Ch. V, \S 4, no. 7]{Bourbaki-Lie456}}] 
 Assume $(W,S)$ is irreducible.
 \begin{enumerate}
   \item The geometric representation $\Vg$ is irreducible if and only if the bilinear form $(-|-)$ defined in Equation \eqref{eq-B-Vg} is non-degenerate.
   \item If $\Vg$ is reducible, then it has a maximal subrepresentation $V_0$ with trivial $W$-action, and the quotient $\Vg / V_0$ is irreducible.
 \end{enumerate}
\end{proposition}

In this subsection we will see that the generalized geometric representations admit a similar criterion.
As an application, we obtain in the next subsection a description of some other reflection representations.

Let $V = \bigoplus_{s \in S} \mathbb{C} \alpha_s$ be a generalized geometric representation of $W$ defined by a GGR-datum, say, $(k_{rt}, a_r^t)_{r,t \in S, r\ne t}$, and $\widetilde{G}$ be the associated graph, with corresponding Coxeter group $\widetilde{W}$ (see Subsection \ref{subsec-tldG}).
Suppose $S = \sqcup_{i \in I} S_i$ is the decomposition of $\widetilde{G}$ into connected components (i.e., for each $S_i$, any two vertices in $S_i$ are connected by a path in $\widetilde{G}$, while for $i \ne j$ and $s \in S_i$, $t \in S_j$, there is no path in $\widetilde{G}$ connecting $s$ and $t$).
Let $W_i$ and $\widetilde{W}_i$ denote the parabolic subgroups of $W$ and $\widetilde{W}$ generated by $S_i$ respectively.
The following lemma is clear.

\begin{lemma}\label{lem-decomp-into-comp}
  Notations as above.
  \begin{enumerate}
    \item $\widetilde{W} = \prod_{i \in I} \widetilde{W}_i$.
    \item The representation $V$ is decomposed into $V = \bigoplus_{i \in I} V_i$, where 
        \[V_i = \bigoplus_{s \in S_i} \mathbb{C} \alpha_s\]
        is a subrepresentation of $W$.
    \item For $j \neq i$, $W_j$ acts trivially on $V_i$.
    \item Each $V_i$ is a generalized geometric representation of the Coxeter group $(W_i, S_i)$, defined by the datum $(k_{rt}, a_r^t)_{r,t \in S_i, r\neq t}$, and factors through $\widetilde{W}_i$.
  \end{enumerate}
\end{lemma}

Thus, we can focus our study on the direct summands $V_i$.

\begin{proposition}\label{prop-reducibility}
Notations as above.
Suppose $\widetilde{G}$ is connected, and $U \subsetneq V$ is a subrepresentation of $W$.
\begin{enumerate}
  \item \label{prop-red-1} The action of $W$ on $U$ is trivial.
  \item \label{prop-red-2} If $V$ is reducible, then $V$ has a maximal subrepresentation $V_0$ with trivial $W$-action, and the quotient $V/V_0$ is irreducible.
  \item \label{prop-red-3} The representation $V$ is indecomposable.
\end{enumerate}
\end{proposition}

\begin{proof}
Suppose $v \in U$ and $s \cdot v \neq v$ for some $s \in S$.
Then $v - s \cdot v \in \mathbb{C}^\times \alpha_s$ by Lemma \ref{lem-refl}.
Thus $\alpha_s \in U$.
If $r \in S$ is adjacent to $s$ in $\widetilde{G}$, then $k_{rt} < \frac{m_{rt}}{2}$. 
We have $\alpha_r \in U$ since $\mathbb{C}\langle \alpha_s, \alpha_r \rangle$ forms an irreducible representation of $\langle s, r \rangle$ isomorphic to $\rho_{k_{rt}}$.
Inductively, for any $t \in S$ we have $\alpha_t \in U$ since $\widetilde{G}$ is connected. 
Thus $U = V$ which contradicts $U \subsetneq V$.
The points \eqref{prop-red-2} and \eqref{prop-red-3} follow from \eqref{prop-red-1}.
\end{proof}

\begin{remark}
  If $\widetilde{G}$ is not assumed to be connected and $V$ is decomposed into $\bigoplus_{i \in I} V_i$ as in Lemma \ref{lem-decomp-into-comp}, then  Proposition \ref{prop-reducibility} can be applied to each  $V_i$.
\end{remark}

Suppose $\widetilde{G}$ is connected.
We label the elements in $S$ so that $$S = \{s_1, s_2, \dots, s_n\},$$ and we use notations $\alpha_i, k_{ij}, a_i^j,$ etc. as in the proof of Lemma \ref{lem-chi1=2}.
Let 
\[\text{$v = \sum_{j=1}^n x_j \alpha_j \in V$ ($x_j \in \mathbb{C}$),}\] 
then
\begin{align*}
    s_i \cdot v & = - x_i \alpha_i + \sum_{j \neq i} x_j \Bigl(\alpha_j + 2 \frac{a_i^j}{a_j^i} \cos \frac{k_{ij} \uppi}{m_{ij}} \alpha_i \Bigr) \\
     & = \Bigl(-x_i + \sum_{j \neq i} 2 x_j \frac{a_i^j}{a_j^i} \cos \frac{k_{ij} \uppi}{m_{ij}}\Bigr) \alpha_i + \sum_{j \neq i} x_j \alpha_j.
\end{align*}
If $v$ is fixed by each $s_i$, then
\begin{equation}  \label{eq6.1}
    x_i - \sum_{j \neq i}  x_j \frac{a_i^j}{a_j^i} \cos \frac{k_{ij} \uppi}{m_{ij}} = 0, \quad \forall i = 1,\dots,n.
\end{equation}
By Proposition \ref{prop-reducibility}, $V$ is reducible if and only if there exists a nonzero vector $v \in V$ fixed by each $s_i$. 
This amounts to saying that the set of equations \eqref{eq6.1} has a nonzero solution for the variables $\{x_i \mid i = 1, \dots, n\}$, which is also equivalent to saying that the following $n \times n$ matrix is singular:
\begin{equation}  \label{matrixA}
    A := \begin{pmatrix}
           1 & - \frac{a_1^2}{a_2^1} \cos \frac{k_{12} \uppi}{m_{12}} & \dots \\
           - \frac{a_2^1}{a_1^2} \cos \frac{k_{12} \uppi}{m_{12}} & 1 &  \\
           \vdots &  & \ddots
         \end{pmatrix}
\end{equation}
(All diagonal elements are $1$, and for any $i \neq j$ the element at $(i,j)$-position is $- \frac{a_i^j}{a_j^i} \cos \frac{k_{ij} \uppi}{m_{ij}}$. Note that $A$ is not symmetric in general.)
The co-rank of $A$ equals to the dimension of the solution space, as well as the dimension of the maximal subrepresentation of $V$.
To conclude:

\begin{theorem}\label{thm-reducibility}
  Let $V$ be a generalized geometric representation of $W$ defined by a GGR-datum, say,  $(k_{rt}, a_r^t)_{r,t \in S, r\ne t}$. 
  Suppose the associated graph $\widetilde{G}$ is connected.
  \begin{enumerate}
    \item \label{6.3-1} The representation $V$ is irreducible if and only if the matrix $A$ defined in Equation \eqref{matrixA} is invertible.
    \item \label{6.3-2} If $V$ is reducible, and $V_0$ is the maximal subrepresentation, then the quotient $V/V_0$ is irreducible of dimension $\operatorname{rank}(A)$.
  \end{enumerate}
\end{theorem}

\begin{example}
  In Example \ref{eg-affine-A}, if we choose $a_0^n = x$ and $a_n^0 = a_0^1 = a_1^0 =  \dots = 1$, then the associated character $\chi$ satisfies $\chi(c) = x$.
  The square matrix $A$ (of size $n + 1$, defined in Equation \eqref{matrixA}) is
  \begin{equation*}
    A = \begin{pmatrix}
        1 & - \frac{1}{2} & 0 & \dots & 0 & - \frac{x}{2} \\
        - \frac{1}{2} & 1 & - \frac{1}{2} &  & & 0 \\
        0 & - \frac{1}{2} & \ddots & \ddots &  & \vdots \\
        \vdots &  & \ddots & \ddots & - \frac{1}{2} & 0 \\
        0 & &  & - \frac{1}{2} & 1 & - \frac{1}{2} \\
        - \frac{1}{2x} & 0 & \dots & 0 & - \frac{1}{2} & 1
      \end{pmatrix}
  \end{equation*}
  One may use induction on $n$ to compute the determinant of $A$,
  \begin{equation*}
    \det A = \frac{2 - x - x^{-1}}{2^{n+1}}.
  \end{equation*}
  Then, by Theorem \ref{thm-reducibility}, this generalized geometric representation is irreducible if and only if $x \neq 1$.
  If $x = 1$, then this representation is the geometric representation $\Vg$ of $\widetilde{\mathsf{A}}_n$.
\end{example}

In general, if $\widetilde{G}$ is not assumed to be connected, then we have the following result by Lemma \ref{lem-decomp-into-comp} and Theorem \ref{thm-reducibility}.

\begin{corollary} \label{cor-reducibility}
  Suppose $V$ is a generalized geometric representation of $W$. Let
  \begin{equation*}
    V_0 := \{v \in V \mid w \cdot v = v, \forall w \in W\}
  \end{equation*}
  be the maximal subrepresentation with trivial $W$-action.
  Then $V/V_0$ is semisimple and $\dim V/V_0 = \rank(A)$.
  In particular, $V$ is semisimple if and only if $A$ is invertible.
\end{corollary}

\subsection{R-representations} \label{subsec-R}
In this subsection we investigate a larger class of reflection representations, relaxing the requirement that reflection vectors form a basis in the definition of generalized geometric representations.
In order to facilitate discussion, we give the following definition.

\begin{definition} \label{def-R}
  Let $V$ be a reflection representation of $W$.
  We call $V$ an R-representation of $W$ if 
  \begin{enumerate}
    \item $V$ is spanned by reflection vectors $\{\alpha_s\mid s\in S\}$;
    \item for any pair $r, t \in S$ $(r \ne t)$ such that $m_{rt} < \infty$, the vectors $\alpha_r, \alpha_t$ are linearly independent.
  \end{enumerate}
\end{definition}

\begin{remark} \label{rmk-a=1}
  The motivation to consider R-representations arises from Kazhdan--Lusztig theory.
  Roughly speaking, if $(W,S)$ is irreducible (i.e., the Coxeter graph is connected), then the R-representations are those reflection representations of $W$ such that any nontrivial composition factor of the representation is a quotient of the cell representation given by the second-highest two-sided cell in the sense of Kazhdan and Lusztig \cite{KL79, Lusztig83-intrep}.
  Moreover, for some specific Coxeter groups, it can be shown that any irreducible representation corresponding to this two-sided cell must be an R-representation.
  We shall discuss this issue and present relevant results in another paper \cite{Hu23}.
  See also \cite{DPWX22} for relevant results about the corresponding based ring $J$, which carries the structure of the cell representation.
\end{remark}

Note that at present we require $m_{rt} < \infty$ for any $r,t$ (see the beginning of Section \ref{sec-IR}), so any two vectors $\alpha_r, \alpha_t$ are not proportional in an R-representation.
By definition, the R-representations form a larger class than the generalized geometric representations.
But in fact, there is not too much difference between them, as we shall show now.

\begin{theorem}\label{thm-R1} \leavevmode
  \begin{enumerate}
    \item \label{5.6.1} Let $V$ be an R-representation of $W$. There is a unique (up to isomorphism) generalized geometric representation $V^\prime$ such that $V$ is a quotient of $V^\prime$, say, $\pi : V^\prime \twoheadrightarrow V$, and $W$ acts on $\ker \pi$ trivially.
    \item \label{5.6.3} The R-representation $V$ is semisimple if and only if  $\ker \pi$ is the maximal subrepresentation of $V^\prime$ with trivial $W$-action.
    \item \label{5.6.2} The isomorphism classes of semisimple R-representations one-to-one correspond to the isomorphism classes of generalized geometric representations, and thus to the set of data $\{((k_{rt})_{r,t \in S, r \ne t}, \chi)\}$ in Theorem \ref{thm-IR}.
        In particular, simple R-representations correspond to those data such that the associated graph $\widetilde{G}$ is connected.
  \end{enumerate}
\end{theorem}

\begin{proof}
  In an R-representation $V = \sum_{s \in S} \mathbb{C} \alpha_s$ of $W$, the subspace $\mathbb{C} \langle \alpha_r, \alpha_t \rangle$ is two dimensional for any $r,t \in S$ with $r \ne t$ by assumption \eqref{eq-fin-condition}.
  Clearly, the conclusions of Lemma \ref{lem4.2} are still valid.
  Therefore, we can extract a GGR-datum $(k_{rt}, a_r^t)_{r,t \in S, r \ne t}$ from the R-representation $V$ as we did in subsection \ref{subsec-struIR}.
  Now let $V^\prime = \bigoplus_{s \in S} \mathbb{C} \alpha_s^\prime$ be the generalized geometric representation of $W$ defined by this datum, and $\pi: V^\prime \twoheadrightarrow V$ be the surjective linear map defined by $\alpha_s^\prime \mapsto \alpha_s$.
  One may easily verify that $\pi$ is a homomorphism of $W$-representations (that is, $s \cdot \pi (\alpha_t^\prime) = \pi(s \cdot \alpha_t^\prime)$).
  Thus $V$ is a quotient of $V^\prime$.
  This proves the ``existence'' part of \eqref{5.6.1}.

  Suppose $V^{\prime \prime} = \bigoplus_{s \in S} \mathbb{C} \alpha_s^{\prime \prime}$ is a generalized geometric representation defined by another  datum $(l_{rt}, b_r^t)_{r,t \in S, r \ne t}$, and suppose there is a  surjective homomorphism $\pi^{\prime \prime}: V^{\prime \prime} \twoheadrightarrow V$ of $W$-representations.
  After rescaling the vectors $\alpha_s^{\prime\prime}$, we may assume $\pi^{\prime \prime}(\alpha_s^{\prime \prime}) = \alpha_s$, $\forall s \in S$.
  For any $r, t \in S$ with $r \ne t$, we must have $k_{rt} = l_{rt}$ since $\mathbb{C} \langle \alpha_r^{\prime \prime}, \alpha_t^{\prime \prime} \rangle$ and $\mathbb{C} \langle \alpha_r, \alpha_t \rangle$ should be isomorphic as representations of $\langle r, t \rangle$.

  Consider the following equalities,
  \begin{align*}
    \alpha_t + 2 \frac{a_r^t}{a_t^r} \cos \frac{k_{rt}\uppi}{m_{rt}} \alpha_r & = r \cdot \alpha_t \\
     & = r \cdot \pi^{\prime\prime} (\alpha_t^{\prime\prime}) \\
     & = \pi^{\prime\prime} (r \cdot \alpha_t^{\prime\prime})\\
     & = \pi^{\prime\prime} (\alpha_t^{\prime\prime} + 2 \frac{b_r^t}{b_t^r} \cos \frac{k_{rt}\uppi}{m_{rt}} \alpha_r^{\prime\prime}) \\
     & = \alpha_t + 2 \frac{b_r^t}{b_t^r} \cos \frac{k_{rt}\uppi}{m_{rt}} \alpha_r.
  \end{align*}
  If $k_{rt} \neq \frac{m_{rt}}{2}$, then we must have $\frac{a_r^t}{a_t^r} = \frac{b_r^t}{b_t^r}$.
  Thus, the two data $(k_{rt}, a_r^t)_{r,t \in S, r \ne t}$ and $(k_{rt}, b_r^t)_{r,t \in S, r \ne t}$ define the same character of $H_1 (\widetilde{G})$, where $\widetilde{G}$ is the graph associated with $(k_{rt})_{r,t \in S, r \ne t}$.
  By Theorem \ref{thm-IR}, we have $V^\prime \simeq V^{\prime \prime}$ as $W$-representations.
  This proves the ``uniqueness'' part of \eqref{5.6.1}.

  Let $S = \sqcup_{i \in I} S_i$ be the decomposition of $\widetilde{G}$ into connected components.
  Then $V^\prime = \bigoplus_{i \in I} V_i^\prime$, where $V_i^\prime = \bigoplus_{s \in S_i} \mathbb{C} \alpha_s^\prime$ is a subrepresentation, as in Lemma \ref{lem-decomp-into-comp}.
  Denote by $V_{i,0}^\prime$ the maximal subrepresentation (may be zero) in $V_i^\prime$.
  The action of $W$ on each $V_{i,0}^\prime$ is the trivial action (see Proposition \ref{prop-reducibility}).
  Note that $\ker \pi$ is a subrepresentation in $V^\prime$ but $\ker \pi$ can not contain any $V_i^\prime$.

  Suppose $0 \neq v \in \ker \pi$.
  We write $v = \oplus_{i \in I} v_i$ where $v_i \in V_i^\prime$.
  Take arbitrarily $j \in I$ and $s \in S_j$.
  For any $i \neq j$, we have $s \cdot v_i = v_i$ by Lemma \ref{lem-decomp-into-comp}.
  If $s \cdot v_j \neq v_j$, then $v_j - s \cdot v_j = v - s \cdot v \in \ker \pi$.
  By Lemma \ref{lem-refl}, we have $\alpha_s^\prime \in \ker \pi$, which is ridiculous.
  Thus $s \cdot v = v$, $\forall s\in S$.
  More specifically, $v_i \in V_{i,0}^\prime$, $\forall i \in I$.
  Hence,  $W$ acts on $\ker \pi$ trivially.
  The proof of \eqref{5.6.1} is complete.

  Let $V_0^\prime := \bigoplus_{i \in I} V_{i,0}^\prime$ be the maximal subrepresentation of $V^\prime$ with trivial $W$-action.
  We have shown that $\ker \pi \subseteq V_0^\prime$.
  By Corollary \ref{cor-reducibility}, the quotient $V^\prime / V_0^\prime$ is a semisimple representation of $W$.
  Moreover, $V^\prime / V_0^\prime$ is an R-representation since $\alpha_s^\prime \notin V_0^\prime$, $\forall s \in S$.
  On the other hand, if $\ker \pi \subsetneq V_0^\prime$, then there exists $j \in I$ such that $V_{j,0}^\prime \nsubseteq \ker \pi$.
  It follows that $\pi(V_j^\prime)$ is an indecomposable but reducible subrepresentation of $V$.
  Thus $V$ is not semisimple.
  This proves \eqref{5.6.3}.

  At last, \eqref{5.6.2} is deduced from \eqref{5.6.1} and \eqref{5.6.3}.
\end{proof}

\begin{remark}
    Let $v = \oplus_i v_i \in \ker \pi$ as in the proof.
    It may happen that $v_i \notin \ker \pi$ for some $i$.
    In fact, for any subspace $K \subseteq \bigoplus_i V_{i,0}^\prime$, the quotient $V^\prime / K$ is an R-representation of $W$.
\end{remark}

\begin{corollary} \label{cor-R1}
  Let \[\text{$\pi_1 : V^\prime \twoheadrightarrow V_1$, $\pi_2 : V^\prime \twoheadrightarrow V_2$}\] as in Theorem \ref{thm-R1}, where $V^\prime$ is a generalized geometric representation of $W$ and $V_1, V_2$ are R-representations.
  Suppose the associated graph $\widetilde{G}$ is connected.
  Then $V_1 \simeq V_2$ if and only if $\ker \pi_1 = \ker \pi_2$.
\end{corollary}

\begin{proof}
  By the same proof of Corollary \ref{cor-end-IR}, we have
  \begin{equation}   \label{eq-cor-R1-1}
    \operatorname{End}_W(V_i) \simeq \mathbb{C}, \text{ and } \operatorname{Hom}_W(V^\prime, V_i) \simeq \mathbb{C}, \text{ for } i = 1,2.
  \end{equation}
  If $\ker \pi_1 = \ker \pi_2$, then obviously $V_1 \simeq V_2$.
  Conversely, suppose  $\phi: V_1 \xrightarrow{\sim} V_2$ is an isomorphism.
  Then $\phi \pi_1: V^\prime \twoheadrightarrow V_2$ is a surjective homomorphism.
  By \eqref{eq-cor-R1-1}, $\phi \pi_1 = c \pi_2$ for some $c \in \mathbb{C}^\times$.
  It follows that $\ker \pi_2  = \ker \pi_1$.
\end{proof}

By Corollary \ref{cor-R1}, if $V^\prime$ is a generalized geometric representation with connected associated graph $\widetilde{G}$, then the isomorphism classes of R-representations which are quotients of $V^\prime$ are parameterized by subspaces of the maximal subrepresentation of $V^\prime$ (a union of Grassmannians).

When $\widetilde{G}$ is not connected, the quotients of a generalized geometric representation by two different subspaces may be isomorphic.
The following is such an example.

\begin{example}
Suppose $W$ is of type $\widetilde{\mathsf{A}}_2 \times \widetilde{\mathsf{A}}_2$, i.e., the Coxeter graph of $W$ is as follows,
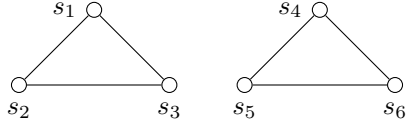
\begin{equation*}
    \begin{tikzpicture}
      \node [circle, draw, inner sep=2pt, label=left:$s_1$] (s1) at (-1.5,1) {};
      \node [circle, draw, inner sep=2pt, label=below:$s_2$] (s2) at (-2.5,0) {};
      \node [circle, draw, inner sep=2pt, label=below:$s_3$] (s3) at (-0.5,0) {};
      \draw (s3) -- (s1) -- (s2) -- (s3);
      \node [circle, draw, inner sep=2pt, label=left:$s_4$] (s4) at (1.5,1) {};
      \node [circle, draw, inner sep=2pt, label=below:$s_5$] (s5) at (0.5,0) {};
      \node [circle, draw, inner sep=2pt, label=below:$s_6$] (s6) at (2.5,0) {};
      \draw (s4) -- (s5) -- (s6) -- (s4);
    \end{tikzpicture}
\end{equation*}
The  geometric representation $\Vg$  is a direct sum of two copies of the geometric representation of $\widetilde{\mathsf{A}}_2$.
Let $v_1 := \alpha_1 + \alpha_2 + \alpha_3$, $v_2 := \alpha_4 + \alpha_5 + \alpha_6$.
Then $V_0 := \mathbb{C}\langle v_1, v_2 \rangle$ is the maximal subrepresentation in $\Vg$ with trivial $W$-action.
The two quotients $\Vg/\mathbb{C} v_1$ and $\Vg/\mathbb{C} v_2$ are clearly non-isomorphic.
However, let $V_1 := \Vg/\mathbb{C} (v_1 + v_2)$, $V_2 := \Vg/\mathbb{C} (v_1 + 2v_2)$, then the map $\alpha_i \mapsto \alpha_i, i = 1,2,3$, $\alpha_j \mapsto 2 \alpha_j, j = 4,5,6$ defines an isomorphism from $V_1$ to $V_2$.
\end{example}

\subsection{\texorpdfstring{$W$}{W}-invariant bilinear form} \label{subsec-bil}
Recall that there is a nonzero bilinear form $(-|-)$ on the geometric representation $\Vg$ of $W$ such that $(-|-)$ stays invariant under the action of $W$ (see Equations \eqref{eq-B-Vg} to \eqref{eq-Vg-W-inv}).
Under this bilinear form, $W$ acts like an orthogonal reflection group on $\Vg$.
Moreover, it can be shown that if $(W,S)$ is irreducible then such a bilinear form is unique up to a scalar.
It is also interesting to ask whether such a bilinear form exists on generalized geometric representations.

\begin{theorem}\label{prop-bil}
  Let $V = \bigoplus_{s \in S} \mathbb{C} \alpha_s$ be a generalized geometric representation of $W$ defined by a GGR-datum $(k_{rt}, a_r^t)_{r,t \in S, r \ne t}$, and $\chi$ be the associated character. Assume the associated graph $\widetilde{G}$ is connected.
  \begin{enumerate}
    \item \label{prop-bil-1} There is a nonzero $W$-invariant bilinear form on $V$ if and only if \[\chi (H_1 (\widetilde{G})) \subseteq \{\pm 1\}.\]
    \item \label{prop-bil-2} There is a nonzero $W$-invariant sesquilinear form on $V$ if and only if \[\chi (H_1 (\widetilde{G})) \subseteq S^1 := \{z \in \mathbb{C} \mid \abs{z} = 1\}.\]
  \end{enumerate}
\end{theorem}

\begin{proof}
Suppose $(-|-)$ is a nonzero $W$-invariant bilinear form on $V$.
For an arbitrary path $(s_1, s_2, \dots, s_n)$ in $\widetilde{G}$, suppose $(\alpha_1 | \alpha_1) = x$ (here $\alpha_i := \alpha_{s_i}$, similarly $a_i^j := a_{s_i}^{s_j}$, $k_{ij} := k_{s_1s_2}$).
Note that $\beta_{s_1} \mapsto a_1^2 \alpha_1$, $\beta_{s_2} \mapsto a_2^1 \alpha_2$ is an isomorphism of $\langle s_1, s_2 \rangle$-representations from $\rho_{k_{12}}$ to the one spanned by $\alpha_1, \alpha_2$.
Then by Lemma \ref{lem-2.7}, we have
\begin{align*}
   (\alpha_1 | \alpha_1) = x & \xLongrightarrow{\phantom{\text{ Lemma \ref{lem-2.7} }}} (a_1^2 \alpha_1 | a_1^2 \alpha_1) = x (a_1^2)^2  \\
   & \xLongrightarrow{\text{ Lemma \ref{lem-2.7} }} (a_2^1 \alpha_2 | a_2^1 \alpha_2) = x(a_1^2)^2 \\
   & \xLongrightarrow{\phantom{\text{ Lemma \ref{lem-2.7} }}} (\alpha_2 | \alpha_2) = x (\frac{a_1^2}{a_2^1})^2.
\end{align*}
Do this argument recursively along the path $(s_1, \dots, s_n)$, then we get
\begin{equation} \label{eq-6.1}
    (\alpha_n|\alpha_n) = x \left( \frac{a_1^2 a_2^3 \cdots a_{n-1}^n}{a_2^1 a_3^2 \cdots a_n^{n-1}} \right)^2.
\end{equation}

We claim that $x \ne 0$.
Otherwise, we have $(\alpha_s|\alpha_s) = 0$ for any $s \in S$ since $\widetilde{G}$ is connected.
As a result, if $s,t \in S$ such that $\widetilde{m}_{st} \ge 3$, then $(\alpha_s|\alpha_t) = (\alpha_t| \alpha_s) = 0$ by Lemma \ref{lem-2.7}.
While if $\widetilde{m}_{st} =2$, then $(\alpha_s|\alpha_t) = (s \cdot \alpha_s| s\cdot \alpha_t) = - (\alpha_s| \alpha_t)$, and hence $(\alpha_s|\alpha_t) = 0$.
Therefore $(-|-) = 0$ which is a contradiction. Thus $x \ne 0$.

Suppose now $s_1 = s_n$. 
Then the path $(s_1, \dots, s_n)$ is a closed path, and $(\alpha_n| \alpha_n) = x$.
Equation \eqref{eq-6.1} implies that
\begin{equation*}
  \chi((s_1, \dots, s_n)) = \frac{a_2^1 a_3^2 \cdots a_n^{n-1}}{a_1^2 a_2^3 \cdots a_{n-1}^n} = \pm 1.
\end{equation*}
Thus, the image of $\chi$ is contained in $\{\pm 1\}$.

Conversely, suppose $\chi (H_1 (\widetilde{G})) \subseteq \{\pm 1\}$.
Choose and fix a vertex $s_1$ in $\widetilde{G}$.
We define $(\alpha_1|\alpha_1) := 1$.
For another vertex $s$, choose a path $(s_1, s_2, \dots, s_n = s)$ in $\widetilde{G}$ connecting $s_1$ and $s$.
In view of Lemma \ref{lem-2.7}, we define
\begin{equation*}
    (\alpha_s|\alpha_s) := \left( \frac{a_1^2 a_2^3 \cdots a_{n-1}^n}{a_2^1 a_3^2 \cdots a_n^{n-1}} \right)^2.
\end{equation*}
We claim that $(\alpha_s|\alpha_s)$ is independent of the choice of the path $(s_1, s_2, \dots, s_n = s)$.
If there is another path, say $(s_1, s_p, s_{p-1}, \dots, s_{n+1}, s_n)$, $p \ge n$, connecting $s_1$ and $s$, then the two paths connected end-to-end form a closed path 
\[(s_1, \dots, s_n, s_{n+1}, \dots, s_p, s_1).\] 
Our assumption says
\begin{equation*}
    \chi ((s_1, \dots, s_n, s_{n+1}, \dots, s_p, s_1)) = \frac{a_2^1 \cdots a_n^{n-1} a_{n+1}^n \cdots a_p^{p-1} a_1^p}{a_1^2 \cdots a_{n-1}^n a_n^{n+1} \cdots a_{p-1}^p a_p^1} = \pm 1.
\end{equation*}
Thus
\begin{equation*}
  \left( \frac{a_1^2 a_2^3 \cdots a_{n-1}^n}{a_2^1 a_3^2 \cdots a_n^{n-1}} \right)^2 = \left( \frac{a_1^p a_p^{p-1} \cdots a_{n+1}^n}{a_p^1 a_{p-1}^p \cdots a_n^{n+1}} \right)^2.
\end{equation*}
This suggests that $(\alpha_s|\alpha_s)$ is independent of the choice of the path.

Moreover, for $t \neq s$, we define
\begin{equation}   \label{eq-bil-2}
    (\alpha_t| \alpha_s) = (\alpha_s|\alpha_t) := - \frac{a_s^t}{a_t^s} \cos \frac{k_{st} \uppi}{m_{st}} (\alpha_s| \alpha_s). 
\end{equation}
(The right hand side also equals to $- \frac{a_t^s}{a_s^t} \cos \frac{k_{st} \uppi}{m_{st}}  (\alpha_t| \alpha_t)$, for, if $\{s,t\}$ is not an edge in $\widetilde{G}$, then $\cos \frac{k_{st} \uppi}{m_{st}} = 0$; if $\{s,t\}$ is an edge in $\widetilde{G}$, then $(s_1, \dots, s_{n-1}, s, t)$ is a path in $\widetilde{G}$ connecting $s_1$ and $t$, and thus $(\alpha_t| \alpha_t) = (\frac{a_s^t}{a_t^s})^2 (\alpha_s| \alpha_s)$.)

Now $(-|-)$ is a well defined bilinear form on $V$. 
We need to check that $(-|-)$ is $W$-invariant. It suffices to check
\begin{equation}  \label{eq-bil-1}
  (s \cdot \alpha_r| s \cdot \alpha_t) =  (\alpha_r| \alpha_t), \quad \forall s, r, t \in S.
\end{equation}
If $s = r = t$, this is obvious.
If $s = r \neq t$, or $s = t \neq r$, or $s \neq r = t$, then everything happens in a dihedral world, and Equation \eqref{eq-bil-1} holds by Lemma \ref{lem-2.7}.
If $s, r, t$ are distinct, then
\begin{align*}
     (s \cdot \alpha_r|s \cdot \alpha_t) & =  \Bigl(\alpha_r + 2 \frac{a_s^r}{a_r^s} \cos \frac{k_{sr} \uppi}{m_{sr}} \alpha_s \Bigm\vert \alpha_t + 2 \frac{a_s^t}{a_t^s} \cos \frac{k_{st} \uppi}{m_{st}} \alpha_s\Bigr) \\
    & =  (\alpha_r| \alpha_t) + 2 \frac{a_s^t}{a_t^s} \cos \frac{k_{st} \uppi}{m_{st}} (\alpha_r| \alpha_s) + 2 \frac{a_s^r}{a_r^s} \cos \frac{k_{sr} \uppi}{m_{sr}}  (\alpha_s| \alpha_t) \\*
    & \mathrel{\phantom{=}} {} + 4 \frac{a_s^t a_s^r}{a_t^s a_r^s} \cos \frac{k_{st} \uppi}{m_{st}} \cos \frac{k_{sr} \uppi}{m_{sr}} (\alpha_s| \alpha_s)\\*
    & =  (\alpha_r| \alpha_t).
\end{align*}
The last equality is by Equation \eqref{eq-bil-2}.

The proof of \eqref{prop-bil-2} is similar (note that we have Lemma \ref{lem-fdih-sesqui}).
\end{proof}

From the proof of Theorem \ref{prop-bil}, we can obtain the following result.

\begin{corollary}
Let $V$ be a generalized geometric representation of $W$.
Suppose the associated graph $\widetilde{G}$ is connected and $V$ admits a nonzero $W$-invariant bilinear form $(-|-)$.
  \begin{enumerate}
    \item Such a bilinear form is symmetric and is unique up to a $\mathbb{C}^\times$-scalar.
    \item For any $s \in S$, the action of $s$ on $V$ is an orthogonal reflection with respect to $(-|-)$, i.e.,
        \begin{equation*}
          s \cdot \alpha_t = \alpha_t - \frac{2(\alpha_t|\alpha_s)}{(\alpha_s | \alpha_s)} \alpha_s, \quad \forall t \in S.
        \end{equation*}
  \end{enumerate}
\end{corollary}

\begin{remark} \leavevmode
  \begin{enumerate}
    \item If $\widetilde{G}$ is not assumed to be connected, and we write $V = \bigoplus_i V_i$ as in Lemma \ref{lem-decomp-into-comp}, then Theorem \ref{prop-bil} can be applied to each $V_i$.
    \item By Theorem \ref{prop-bil}, there are only finitely many  generalized geometric representations of $W$ admitting a nonzero $W$-invariant bilinear form. The geometric representation $\Vg$ is one of them.
  \end{enumerate}
\end{remark}

In the rest of this subsection, we discuss bilinear forms on R-representations.
Let $V$ be a generalized geometric representation with connected $\widetilde{G}$.
Suppose $V$ is not irreducible, and $V_0$ is the maximal subrepresentation.

Assume first that $V$ admits a nonzero $W$-invariant bilinear form $(-|-)$.
Let $v = \sum_{s \in S} x_s \alpha_s \in V_0$ ($x_s \in \mathbb{C}$).
Then
\begin{equation} \label{eq-6.4}
  \begin{aligned}
     (v| \alpha_r) & = x_r  (\alpha_r| \alpha_r) + \sum_{s \neq r} x_s (\alpha_s| \alpha_r) \\
     & = x_r  (\alpha_r| \alpha_r) - \sum_{s \neq r} x_s \frac{a_r^s}{a_s^r} \cos \frac{k_{sr} \uppi}{m_{sr}} (\alpha_r| \alpha_r) \\
      &= 0, \quad \forall r \in S.
  \end{aligned}
\end{equation}
The second equality is due to Lemma \ref{lem-2.7} and the third one is due to Equation \eqref{eq6.1}.
By Equation \eqref{eq-6.4}, we see that $V_0$ is contained in the radical of the symmetric bilinear form $(-|-)$.

Now let $V_1$ be an R-representation which is a quotient of $V$ and $\pi : V \twoheadrightarrow V_1$ be the projection. Then $\ker \pi \subseteq V_0$.
For any $x \in V_1$, denote by $\widetilde{x}$ an arbitrary vector in $V$ such that $\pi(\widetilde{x}) = x$.
Equation \eqref{eq-6.4} tells us that $(x|y)^\prime := (\widetilde{x}|\widetilde{y})$ is a well defined bilinear form on $V_1$.
Clearly, $(-|-)^\prime$ is also $W$-invariant.

Conversely, let $V_1$ be an R-representation with connected associated graph $\widetilde{G}$, and $V$ be the corresponding  generalized geometric representation.
Suppose $V_1$ has a nonzero $W$-invariant bilinear form $(-|-)$.
By pulling back alone the surjection $V \twoheadrightarrow V_1$, we obtain a nonzero $W$-invariant bilinear form on $V$.

To conclude, we have the following proposition.

\begin{proposition}
  Let $V_1$ be an R-representation of $W$ and $V$ be the corresponding  generalized geometric representation.
  Suppose the associated graph $\widetilde{G}$ is connected.
  Then $V_1$ admits a nonzero $W$-invariant bilinear form if and only if $V$ admits such a bilinear form.
\end{proposition}

\subsection{Dual of generalized geometric representations} \label{subsec-dual}
In this subsection we consider the dual representation of a generalized geometric representation $V$.
Roughly speaking, if $V$ corresponds to datum $((k_{rt})_{r,t \in S, r \ne t}, \chi)$, then its dual $V^*$ is either a generalized geometric representation or contains an R-representation corresponding to the datum $((k_{rt})_{r,t \in S, r \ne t}, \chi^*)$.
Here $\chi^*$ is the dual character, $\chi^* (-) = \chi (-)^{-1}$.

In the rest of this subsection, let $V = \bigoplus_{s \in S} \mathbb{C} \alpha_s$ be a generalized geometric representation of $W$ defined by a GGR-datum $(k_{rt}, a_r^t)_{r,t \in S, r \ne t}$, and $\{\alpha_s^* \mid s \in S\}$ be the dual basis of $\{\alpha_s \mid s \in S\}$ of $V^* := \operatorname{Hom}_\mathbb{C} (V, \mathbb{C})$.
The dual space $V^*$ forms a representation of $W$ via
\begin{equation*}
  (s \cdot \alpha_t^*) (-) := \alpha_t^*( s \cdot -), \quad \forall s, t \in S.
\end{equation*}
A direct computation shows
\begin{align}
    s \cdot \alpha_s^* & = - \alpha_s^* + 2 \sum_{t \neq s}  \frac{a_s^t}{a_t^s} \cos \frac{k_{st} \uppi}{m_{st}} \alpha_t^*, \quad \forall s \in S; \label{eq6.5} \\
    s \cdot \alpha_t^* & = \alpha_t^*, \quad \forall s,t \in S,  t \neq s. \label{eq6.6}
\end{align}
From this, we can see that the action of $s$ on $V^*$ is a reflection, with  a reflection vector
\begin{equation} \label{eq-6.5}
    \gamma_s := \alpha_s^* - \sum_{t \neq s} \frac{a_s^t}{a_t^s} \cos \frac{k_{st} \uppi}{m_{st}} \alpha_t^*.
\end{equation}

\begin{lemma} \label{lem-dual-dih}
  Notations as above.
  \begin{enumerate}
    \item \label{lem-dual-1} For $r,t \in S$ with $r \ne t$, we have $r \cdot \gamma_t = \gamma_t + 2 \frac{a_t^r}{a_r^t} \cos \frac{k_{rt} \uppi}{m_{rt}} \gamma_r$.
    \item \label{lem-dual-2} For $r,t \in S$ with $r \ne t$, the vectors $\gamma_r$ and $\gamma_t$ are linearly independent.
  \end{enumerate}
\end{lemma}

\begin{proof}
  The proof of \eqref{lem-dual-1} is just a direct computation using Equations \eqref{eq6.5}, \eqref{eq6.6} and \eqref{eq-6.5},
  \begin{align}
    r \cdot \gamma_t & = r \cdot \Bigl( \alpha_t^* - \sum_{s \neq t} \frac{a_t^s}{a_s^t} \cos \frac{k_{st} \uppi}{m_{st}} \alpha_s^* \Bigr) \notag \\
    & = \Bigl( \alpha_t^* - \sum_{s \neq r, t} \frac{a_t^s}{a_s^t} \cos \frac{k_{st} \uppi}{m_{st}} \alpha_s^* \Bigr) - \frac{a_t^r}{a_r^t} \cos \frac{k_{rt} \uppi}{m_{rt}} (r \cdot \alpha_r^*) \notag \\
    & = \Bigl( \alpha_t^* - \sum_{s \neq r, t} \frac{a_t^s}{a_s^t} \cos \frac{k_{st} \uppi}{m_{st}} \alpha_s^* \Bigr) \notag  - \frac{a_t^r}{a_r^t} \cos \frac{k_{rt} \uppi}{m_{rt}} \Bigl( - \alpha_r^* + 2 \sum_{s \neq r}  \frac{a_r^s}{a_s^r} \cos \frac{k_{sr} \uppi}{m_{sr}} \alpha_s^* \Bigr)  \notag \\
    & = \Bigl( \alpha_t^* - \sum_{s \neq  t} \frac{a_t^s}{a_s^t} \cos \frac{k_{st} \uppi}{m_{st}} \alpha_s^* \Bigr) \notag  +
    \frac{a_t^r}{a_r^t} \cos \frac{k_{rt} \uppi}{m_{rt}} \Bigl( 2 \alpha_r^* - 2 \sum_{s \neq r}  \frac{a_r^s}{a_s^r} \cos \frac{k_{sr} \uppi}{m_{sr}} \alpha_s^* \Bigr)  \notag \\
    & = \gamma_t + 2 \frac{a_t^r}{a_r^t} \cos \frac{k_{rt} \uppi}{m_{rt}} \gamma_r.   \label{eq-lem-dual}
  \end{align}

  Now we prove \eqref{lem-dual-2}.
  We have
  \begin{align*}
    \gamma_r & = \alpha_r^* - \frac{a_r^t}{a_t^r} \cos \frac{k_{rt} \uppi}{m_{rt}} \alpha_t^* - \sum_{s \neq r, t} \frac{a_r^s}{a_s^r} \cos \frac{k_{sr} \uppi}{m_{sr}} \alpha_s^*, \\
    \gamma_t & = \alpha_t^* - \frac{a_t^r}{a_r^t} \cos \frac{k_{rt} \uppi}{m_{rt}} \alpha_r^* - \sum_{s \neq r, t} \frac{a_t^s}{a_s^t} \cos \frac{k_{st} \uppi}{m_{st}} \alpha_s^*.
  \end{align*}
  Note that $\{\alpha_s^* \mid s \in S\}$ is a basis of $V^*$.
  If $\gamma_r$ and $\gamma_t$ are proportional, then
  \begin{equation*}
    1 = \left(\frac{a_r^t}{a_t^r} \cos \frac{k_{rt} \uppi}{m_{rt}} \right) \left( \frac{a_t^r}{a_r^t} \cos \frac{k_{rt} \uppi}{m_{rt}} \right) = \cos^2 \frac{k_{rt} \uppi}{m_{rt}}
  \end{equation*}
  which is impossible since $1 \le k_{rt} \le \frac{m_{rt}}{2} < \infty$ (currently we assume that $m_{rt} < \infty$, see assumption \eqref{eq-fin-condition}).
\end{proof}

\begin{theorem} \label{thm-dual}
  Notations as above.
  \begin{enumerate}
    \item \label{thm-dual-1} The vectors $\{\gamma_s \mid s \in S\}$ are linearly independent if and only if the matrix $A$ defined in Equation \eqref{matrixA} is invertible (equivalently, if and only if $V$ is semisimple, see Theorem \ref{thm-reducibility}).
    \item \label{thm-dual-2} If $V$ is semisimple, then $V^* = \bigoplus_{s} \mathbb{C} \gamma_s$, and $V^*$ is a generalized geometric representation of $W$.
        Suppose $((k_{rt})_{r,t \in S, r \ne t},\chi)$ is the datum corresponding to  $V$.
        Then $V^*$ corresponds to the datum $((k_{rt})_{r,t\in S, r \ne t},\chi^*)$.
    \item \label{thm-dual-3} In general cases ($V$ is not assumed to be semisimple), $\{\gamma_s \mid s \in S\}$ span a subrepresentation $V_1 := \sum_{s \in S} \mathbb{C} \gamma_s$ of $V^*$.
        Moreover, the subrepresentation $V_1$ is an R-representation isomorphic to the semisimple quotient of the  generalized geometric representation corresponding to the datum \[((k_{rt})_{r,t \in S, r \ne t}, \chi^*).\]
        The action of $W$ on the quotient $V^* / V_1$ is trivial.
  \end{enumerate}
\end{theorem}

\begin{proof}
  By Formula \eqref{eq-6.5}, the transition matrix from $\{\alpha_s^* \mid s \in S\}$ to $\{\gamma_s \mid s \in S\}$ is $A^\mathrm{T}$, the transpose of $A$.
  Thus $\{\gamma_s\mid s \in S\}$ are linearly independent if and only if $A$ is invertible.
  This proves \eqref{thm-dual-1}.

  If $V$ is semisimple, then $\{\gamma_s \mid s \in S\}$ are linearly independent by \eqref{thm-dual-1}, and clearly they form a basis of $V^*$.
  Note that for any $s \in S$, $\gamma_s$ is the unique $-1$-eigenvector of $s$.
  Thus $V^*$ is a generalized geometric representation.
  By comparing Equations \eqref{eq-lem-dual} and \eqref{eq-IR-action}, we see that $V^*$ is defined by the GGR-datum $(k_{rt}, b_r^t)_{r,t \in S, r \ne t}$ where $b_r^t = a_t^r$.
  Thus, $\chi^*$ is the character of $H_1(\widetilde{G})$ associated with $V^*$.
  This proves \eqref{thm-dual-2}.

  Now we drop the assumption that $V$ is semisimple.
  By Lemma \ref{lem-dual-dih}, $V_1 = \sum_{s \in S} \mathbb{C} \gamma_s$ is a subrepresentation of $V^*$.
  Moreover, $\gamma_r$ and $\gamma_t$ are not proportional for any $r,t \in S$ with $r \ne t$.
  Thus $V_1$ is an R-representation.

  Let $V^\prime$ be the  generalized geometric representation of $W$ such that $V_1$ is a quotient of $V^\prime$.
  Equation \eqref{eq-lem-dual} tells us that $V^\prime$ is defined by the GGR-datum $(k_{rt}, b_r^t)_{r,t \in S, r \ne t}$ where $b_r^t = a_t^r$.
  As in \eqref{thm-dual-2}, the representation $V^\prime$ corresponds to the datum $((k_{rt})_{r,t \in S, r \ne t}, \chi^*)$.
  Let
  \begin{equation*}
    V_0 := \{v \in V^\prime \mid w \cdot v = v, \forall w \in W\}.
  \end{equation*}
  Then by Corollary \ref{cor-reducibility}, $\dim V^\prime/V_0 = \rank(A^\mathrm{T})$ since $A^\mathrm{T}$ is the corresponding matrix of $V^\prime$ defined as in Equation \eqref{matrixA}.
  From the proof of \eqref{thm-dual-1}, we know that $\dim V_1 = \rank (A^\mathrm{T}) = \dim V^\prime/V_0$.
  In view of Theorem \ref{thm-R1}, $V_1$ must be semisimple and isomorphic to $V^\prime/V_0$.

  By Equations \eqref{eq6.5} and \eqref{eq-6.5}, we have
  \begin{equation}   \label{eq-thm-dual}
    s \cdot \alpha_s^* = \alpha_s^* -2 \gamma_s.
  \end{equation}
  By Equations \eqref{eq6.6} and \eqref{eq-thm-dual}, we see that $W$ acts on $V^*/V_1$ trivially.
  The proof of \eqref{thm-dual-3} is complete.
\end{proof}

\section{General situation: allowing \texorpdfstring {$m_{rt}$}{mst} to be \texorpdfstring{$\infty$}{infinity}} \label{sec-gen}

From now on, we remove the assumption \eqref{eq-fin-condition} that $m_{rt} < \infty$, $\forall r,t \in S$.

\subsection{Generalized geometric representations and R-representations}
For any pair $r, t \in S$ with $r \ne t$, we define a parameter set
\begin{equation*}
    P_{rt} = P_{tr} := \begin{cases}
            \{\rho_k \mid k \in \mathbb{N}, 1 \leq k \leq \frac{m_{rt}}{2}\}, & \mbox{if } m_{rt} < \infty; \\
            \{\varrho_r^t, \varrho_t^r\} \cup \{\varrho_z \mid z \in \mathbb{C}\}, & \mbox{if } m_{rt} = \infty.
          \end{cases}
\end{equation*}
(The notations $\varrho_r^t, \varrho_t^r, \varrho_z$ are defined in Subsection \ref{subsec-inf-dih}.)
As in Section \ref{sec-pre}, we use the notations $\beta_r, \beta_t$ to denote the basis of the representation space of an representation in $P_{rt}$.
If $m_{rt} = \infty$, then, by Lemma \ref{lem-infdih-ind-rep} and the fact that $\varrho_0 \simeq \varepsilon_r \oplus \varepsilon_t$, the set $P_{rt}$ consists of all the generalized geometric representations of the dihedral subgroup $\langle r,t \rangle \simeq D_\infty$.

Suppose for any pair $r, t \in S$ with $r\ne t$, we are given $\delta_{rt} \in P_{rt}$ such that $\delta_{rt} = \delta_{tr}$ indicating the same representation of $\langle r,t \rangle$.
Define a Coxeter group $(\widetilde{W},S)$ with Coxeter matrix $(\widetilde{m}_{rt})_{r,t \in S}$ as follows.

\begin{enumerate}
  \item If $m_{rt} < \infty$ and $\delta_{rt} = \rho_{k_{rt}}$, then  $\widetilde{m}_{rt} := \frac{m_{rt}}{d_{rt}}$ where $d_{rt} = \gcd (m_{rt}, k_{rt})$.
  \item If $m_{rt} = \infty$, $\delta_{rt} = \varrho_z$ and $z = 4 \cos^2 \frac{k \uppi}{m}$ for some $k,m \in \mathbb{N}$ such that $k,m$ are co-prime and $1 \leq k < \frac{m}{2}$, then $\widetilde{m}_{rt} := m$.
  \item If $m_{rt} = \infty$, $\delta_{rt} = \varrho_0$, then $\widetilde{m}_{rt} := 2$.
  \item Otherwise, let $\widetilde{m}_{rt} := \infty$.
\end{enumerate}
We denote by $\widetilde{G}$ the Coxeter graph of $\widetilde{W}$, and we say $\widetilde{G}$ is the \emph{associated graph} of $(\delta_{rt})_{r,t \in S, r \ne t}$.
Again, $\widetilde{G}$ and the Coxeter graph $G$ of $W$ have the same vertex set $S$, and the edge set of $\widetilde{G}$  is a subset of that of $G$ (forgetting the labels on the edges).
An edge $\{r,t\}$ in $G$ is an edge in $\widetilde{G}$ if and only if either one of the following is satisfied: (1) $m_{rt} < \infty$ and $\delta_{rt} \ne \rho_{\frac{m_{rt}}{2}}$; (2) $m_{rt} = \infty$ and $\delta_{rt} \ne \varrho_0$.

The following lemma is crucial in the classification of the generalized geometric representations of $W$.

\begin{lemma} \label{lem-gen} \leavevmode
  \begin{enumerate}
    \item \label{lem-gen-1} For any $r,t \in S$ with $r \ne t$, the representations in $P_{rt}$ are non-isomorphic to each other.
    \item \label{lem-gen-2} Given $(\delta_{rt})_{r,t \in S, r \ne t}$ where $\delta_{rt} = \delta_{tr} \in P_{rt}$, let $\widetilde{G}$ be its associated graph.
        If $\{r,t\}$ is an edge in $\widetilde{G}$, then any endomorphism of the representation $\delta_{rt}$ is a scalar.
  \end{enumerate}
\end{lemma}

\begin{proof}
  The point \eqref{lem-gen-1} is due to Lemma \ref{lem-Dm-rep} and Lemma \ref{lem-infdih-ind-rep}.
  The point \eqref{lem-gen-2} is derived from Lemma \ref{lem-schur} and Lemma \ref{lem-2.11}\eqref{lem-infdih-schur}.
\end{proof}

Define a GGR-datum of $W$ to be a set $(\delta_{rt}, a_r^t)_{r , t \in S, r \ne t}$ in which $a_r^t \in \mathbb{C}^\times$ and $\delta_{rt} \in P_{rt}$ such that $\delta_{rt} = \delta_{tr}$.
The character $\chi$ of $H_1(\widetilde{G})$ associated with the datum $(\delta_{rt}, a_r^t)_{r , t \in S, r \ne t}$ is defined by \eqref{eq-4.3} as before.

The classification of the generalized geometric representations of $W$ is nearly the same as Theorem \ref{thm-IR}.

\begin{theorem} \label{thm-IR-gen}
  The isomorphism classes of generalized geometric representations of $(W,S)$ one-to-one correspond to the set of data
  \begin{equation*}
    \left\{ \bigl((\delta_{rt})_{r, t \in S, r \ne t}, \chi\bigr) \text{ } \middle\vert 
    \begin{gathered}
      \text{$\delta_{rt} = \delta_{tr} \in P_{rt}$, $\forall r, t \in S$, $r \ne t$;} \\
      \text{$\chi$ is a character of  $H_1(\widetilde{G}$), where} \\
      \text{ $\widetilde{G}$ is the associated graph of $(\delta_{rt})_{r,t \in S, r \ne t}$}
    \end{gathered}\right\}.
  \end{equation*}
\end{theorem}

\begin{proof}
We basically follow the line of the proofs in Section \ref{sec-IR}.
Let $V = \bigoplus_{s \in S} \mathbb{C} \alpha_s$ be a generalized geometric representation of $W$.
Clearly, Lemma \ref{lem4.2}\eqref{4.2-1} is still valid, i.e., for any pair $r, t \in S$ with $r \ne t$, the subspace $\mathbb{C} \langle \alpha_r, \alpha_t \rangle$ forms a two dimensional representation $\phi_{rt}$ of the subgroup $\langle r,t \rangle$.
Furthermore, $\phi_{rt}$ is a generalized geometric representation of $\langle r,t \rangle$.
Thus, there is a representation $\delta_{rt} \in P_{rt}$ such that  $\delta_{rt} \simeq \phi_{rt}$.
This isomorphism must send $\beta_r$ to $a_r^t \alpha_r$ and $\beta_t$ to $a_t^r \alpha_t$ for some $a_r^t, a_t^r \in \mathbb{C}^\times$.
Similar to the arguments in Subsection \ref{subsec-struIR}, we see that the action of $W$ on the generalized geometric representation $V$ is determined by a GGR-datum $(\delta_{rt}, a_r^t)_{r,t \in S, r \ne t}$.

Conversely, given arbitrarily a GGR-datum $(\delta_{rt}, a_r^t)_{r,t \in S, r \ne t}$, we can define a linear map $\phi(r)$ for any $r \in S$ on the vector space $V := \bigoplus_{s \in S} \mathbb{C} \alpha_s$ as follows.
Define $\phi(r) \cdot \alpha_r := -\alpha_r$.
For $t \in S \setminus \{r\}$, if $\delta_{rt}(r) \cdot \beta_t = \beta_t + c\beta_r$ where $c \in \mathbb{C}$, then we define $\phi(r) \cdot \alpha_t : = \alpha_t + c \frac{a_r^t}{a_t^r} \alpha_r$.
The same arguments as in the proof of Lemma \ref{lem-defrep} show that $V$ is a generalized geometric representation of $W$ under the action $\phi$.

By Lemma \ref{lem-gen}\eqref{lem-gen-1}, we have an analogue of Lemma \ref{lem-k=l}, stated as follows:
if two GGR-data of $W$, say $(\delta_{rt}, a_r^t)_{r,t \in S, r \ne t}$ and $(\zeta_{rt}, b_r^t)_{r,t \in S, r \ne t}$, define isomorphic generalized geometric representations, then $\delta_{rt} = \zeta_{rt}$ for any $r,t \in S$ with $r \ne t$.

Notice that we have Lemma \ref{lem-gen}\eqref{lem-gen-2} which serves as Schur's lemma.
Therefore, the same arguments as in the proofs of Lemma \ref{lem-chi1=2} and Lemma \ref{lem-inj} show that the two  data  $(\delta_{rt}, a_r^t)_{r,t \in S, r \ne t}$ and $(\delta_{rt}, b_r^t)_{r,t \in S, r \ne t}$ define isomorphic generalized geometric representations if and only if they share the same associated character of $H_1(\widetilde{G})$.

At last, Lemma \ref{lem-surj} still works, showing that for any set $(\delta_{rt})_{r,t \in S, r \ne t}$ and any character $\chi: H_1(\widetilde{G}) \to \mathbb{C}^\times$, there exists a GGR-datum $(\delta_{rt}, a_r^t)_{r,t \in S, r \ne t}$ such that $\chi$ is the associated character.
The theorem follows.
\end{proof}

We have a description of R-representations similar to Theorem \ref{thm-R1}\eqref{5.6.1}.

\begin{theorem}\label{thm-R1-gen}
  Let $V = \sum_{s \in S} \mathbb{C} \alpha_s$ be an R-representation of $W$.
  Then there is a unique generalized geometric representation $V^\prime = \bigoplus_{s \in S} \mathbb{C} \alpha_s^\prime$ such that $V$ is a quotient of $V^\prime$, say, $\pi : V^\prime \twoheadrightarrow V$.
  Moreover, $W$ acts on $\ker \pi$ trivially.
\end{theorem}

\begin{proof}
  First, as in the proof of Theorem \ref{thm-R1}, we want to construct a GGR-datum of $W$ from the R-representation $V$.

  For any $r,t \in S$ with $r \ne t$, if $\alpha_r$ and $\alpha_t$ are linearly independent, then $\mathbb{C} \langle \alpha_r, \alpha_t \rangle$ is a generalized geometric representation of $\langle r,t \rangle$, and then we obtain an element $\delta_{rt} \in P_{rt}$ and numbers $a_r^t, a_t^r \in\mathbb{C}^\times$ as before.

  Note that currently $m_{rt}$ is allowed to be $\infty$, so $\alpha_r$ and $\alpha_t$ may be proportional when $m_{rt} = \infty$.
  If this happens, then we choose $\delta_{rt} = \varrho_4$, the geometric representation of $\langle r,t \rangle \simeq D_\infty$, and we choose $a_r^t, a_t^r \in \mathbb{C}^\times$ such that $a_r^t \alpha_r + a_t^r \alpha_t = 0$.

  Now we obtain a GGR-datum $(\delta_{rt}, a_r^t)_{r,t \in S, r \ne t}$.
  Let $V^\prime = \bigoplus_{s \in S} \mathbb{C} \alpha_s^\prime$ be the generalized geometric representation of $W$ defined by this datum, and $\pi: V^\prime \twoheadrightarrow V$ be the linear map $\alpha_s^\prime \mapsto \alpha_s$, $\forall s \in S$.
  We need to verify that $\pi$ is a homomorphism of $W$-representations, i.e., verify that
  \begin{equation}   \label{eq-thm-R1-gen-1}
   \pi(r \cdot \alpha_t^\prime) = r \cdot \pi (\alpha_t^\prime), \quad \forall r,t \in S.
  \end{equation}
  This is clear if $r = t$.
  If $\dim \mathbb{C} \langle \alpha_r, \alpha_t \rangle = 2$, then the verification is also routine as before.
  If $m_{rt} = \infty$ and $\alpha_r, \alpha_t$ are proportional, then we have
  \begin{align*}
    r \cdot \pi (\alpha_t^\prime) & = r \cdot \alpha_t = - \alpha_t, \\
    \pi(r \cdot \alpha_t^\prime) & = \pi(\alpha_t^\prime + 2 \frac{a_r^t}{a_t^r} \alpha_r^\prime) = \alpha_t + 2 \frac{a_r^t}{a_t^r} \alpha_r = - \alpha_t.
  \end{align*}
  The last equality uses the assumption that $a_r^t \alpha_r + a_t^r \alpha_t = 0$.
  Therefore, Equation  \eqref{eq-thm-R1-gen-1} is always true, and $\pi$ is a homomorphism of $W$-representations.

  Suppose $\pi^{\prime\prime}: V^{\prime\prime} \twoheadrightarrow V$ is another surjection of $W$-representations, where $V^{\prime\prime} = \bigoplus_{s \in S} \mathbb{C} \alpha_s^{\prime\prime}$ is a generalized geometric representation defined by a GGR-datum $(\zeta_{rt}, b_r^t)_{r,t \in S, r \ne t}$.
  After rescaling the vectors $\alpha_s^{\prime\prime}$, we may assume $\pi^{\prime \prime}(\alpha_s^{\prime \prime}) = \alpha_s$, $\forall s \in S$.
  For any $r,t \in S$ with $r \ne t$, if $\alpha_r$ and $\alpha_t$ are linearly independent, then $\pi^{\prime\prime}|_{\mathbb{C} \langle \alpha_r^{\prime\prime}, \alpha_t^{\prime\prime} \rangle}$ is an isomorphism of $\langle r,t \rangle$-representations, and then $\zeta_{rt} = \delta_{rt}$.
  If $\alpha_r$ and $\alpha_t$ are proportional, then $\pi^{\prime\prime}|_{\mathbb{C} \langle \alpha_r^{\prime\prime}, \alpha_t^{\prime\prime} \rangle}$ is a surjective homomorphism form $\zeta_{rt}$ to the sign representation $\varepsilon$ of $\langle r,t \rangle$.
  Therefore, $\zeta_{rt}$ must be $\varrho_4$, and thus $\zeta_{rt} = \delta_{rt}$.
  It follows that the two data $(\delta_{rt}, a_r^t)_{r,t \in S, r \ne t}$ and $(\zeta_{rt}, b_r^t)_{r,t \in S, r \ne t}$ have the same associated graph $\widetilde{G}$.

  Suppose now $\{r,t\}$ is an edge in $\widetilde{G}$.
  If $\delta_{rt} \ne \varrho_r^t$, then we consider the equality $\pi^{\prime\prime} (r \cdot \alpha_t^{\prime\prime}) = r \cdot \pi^{\prime\prime} (\alpha_t^{\prime\prime})$ as in the proof of Theorem \ref{thm-R1}, and then we obtain $\frac{a_r^t}{a_t^r} = \frac{b_r^t}{b_t^r}$.
  If $\delta_{rt} = \varrho_r^t$, then we consider instead the equality $\pi^{\prime\prime} (t\cdot \alpha_r^{\prime\prime}) = t \cdot \pi^{\prime\prime} (\alpha_r^{\prime\prime})$ to obtain the same result.
  Anyway, $(\delta_{rt}, a_r^t)_{r,t \in S, r \ne t}$ and $(\delta_{rt}, b_r^t)_{r,t \in S, r \ne t}$ have the same associated character of $H_1(\widetilde{G})$.
  Thus, the two  generalized geometric representations $V^\prime$ and $V^{\prime\prime}$ are isomorphic.

  At last, if the action of $W$ on $\ker \pi$ is not trivial, say $s \cdot v \ne v$ for some $s \in S$, $v \in \ker \pi$, then $\alpha_s \in \ker \pi$ which is a contradiction.
\end{proof}

\subsection{The matrix \texorpdfstring{$A$}{A}}

Let $V = \bigoplus_{s \in S} \mathbb{C} \alpha_s$ be a generalized geometric representation of $W$ defined by GGR-datum, say, $(\delta_{rt}, a_r^t)_{r,t \in S, r \ne t}$.
In consideration of Theorem \ref{thm-R1-gen}, we wish to find subrepresentations of $V$ with trivial $W$-action.
We write $S = \{s_1, \dots, s_n\}$, and use notations $P_{ij}, \delta_{ij}, \varrho_i^j$, etc. in the same way as before.

For $i = 1, \dots, n$, we define
\begin{align*}
    S_1(i) & := \{s_j \in S \mid j \ne i, m_{ij} < \infty\}, \\
    S_2(i) & := \{s_j \in S \mid m_{ij} = \infty, \delta_{ij} = \varrho_j^i \}, \\
    S_3(i) & := \{s_j \in S \mid m_{ij} = \infty, \delta_{ij} = \varrho_i^j \}, \\
    S_4(i) & := \{s_j \in S \mid m_{ij} = \infty, \delta_{ij} = \varrho_{z_{ij}}, z_{ij} \in \mathbb{C}\}.
\end{align*}
Then we have $S = S_1(i) \sqcup S_2(i) \sqcup S_3(i) \sqcup S_4(i) \sqcup \{s_i\}$.
If $s_j \in S_1(i)$, then we denote by $k_{ij}$ the natural number such that $\delta_{ij} = \rho_{k_{ij}}$.
If $s_j \in S_4(i)$, then we denote by $u_{ij}  : = u(z_{ij})$ the complex number chosen in Subsection \ref{subsec-inf-dih}.

Let $v = \sum_{j = 1}^n x_j \alpha_j \in V$ where each $x_j \in \mathbb{C}$.
Then we have

\begin{equation}   \label{eq-A-gen}
  \begin{aligned}
    s_i \cdot v & = -x_i \alpha_i + \sum_{s_j \in S_1(i)} x_j \Bigl(\alpha_j + 2 \frac{a_i^j}{a_j^i} \cos \frac{k_{ij} \uppi}{m_{ij}} \alpha_i\Bigr) + \sum_{s_j \in S_2(i)} x_j \alpha_j  \\
    & \mathrel{\phantom{=}} + \sum_{s_j \in S_3(i)} x_j \Bigl(\alpha_j + \frac{a_i^j}{a_j^i} \alpha_i \Bigr) + \sum_{s_j \in S_4(i)} x_j \Bigl( \alpha_j + u_{ij} \frac{a_i^j}{a_j^i} \alpha_i\Bigr) \\
    & = \Bigl(- x_i + \sum_{s_j \in S_1(i)} 2x_j \frac{a_i^j}{a_j^i} \cos \frac{k_{ij} \uppi}{m_{ij}} \\
    & \mathrel{\phantom{=}} \phantom{\Bigl(} + \sum_{s_j \in S_3(i)} x_j \frac{a_i^j}{a_j^i}  + \sum_{s_j \in S_4(i)} x_j u_{ij} \frac{a_i^j}{a_j^i} \Bigr) \alpha_i + \sum_{j \neq i} x_j \alpha_j. 
  \end{aligned}
\end{equation}
From Equation \eqref{eq-A-gen}, we see that $v$ stays invariant under the action of $W$ if and only if the coordinate vector $\boldsymbol{x} := (x_1, \dots, x_n)^\mathrm{T}$ satisfies the equation $A \cdot \boldsymbol{x} = 0$ where the matrix $A = (A_{ij})$ is defined by (here $A_{ij}$ denotes the entry at $(i,j)$-position of the matrix $A$)
\begin{equation*}
    A_{ij} = \begin{cases}
      1, & \mbox{if } i = j, \\
      - \frac{a_i^j}{a_j^i} \cos \frac{k_{ij} \uppi}{m_{ij}}, & \mbox{if } s_j \in S_1(i), \\
      0, & \mbox{if } s_j \in S_2(i), \\
      - \frac{a_i^j}{2 a_j^i}, & \mbox{if } s_j \in S_3(i), \\
      - \frac{u_{ij} a_i^j}{2 a_j^i}, & \mbox{if } s_j \in S_4(i).
    \end{cases}
\end{equation*}
As in Subsection \ref{subsec-red}, the null space of $A$ is the subspace of $V$ with trivial $W$-action.

\subsection{Other reducibility} \label{subsec-red-gen}
Unlike Proposition \ref{prop-reducibility}, there are some different phenomena when considering reducibility of generalized geometric representations without the assumption $m_{st}< \infty, \forall s,t \in S$.
For example, the representations $\varrho_t^r$ and $\varrho_r^t$ of $D_\infty$ have subrepresentations with nontrivial group action.

Suppose $V = \bigoplus_{s \in S} \mathbb{C} \alpha_s$ is a generalized geometric representation of $W$.
For any pair $r,t \in S$ with $r \neq t$, if $\delta_{rt} = \varrho_r^t$, then we assign a direction $t \to r$ to the edge $\{r,t\}$.
Symmetrically, we have $r \to t$ if $\delta_{rt} = \varrho_t^r$.
For other edges $\{r,t\}$ in the associated graph $\widetilde{G}$, both directions are assigned, i.e., $r \to t$ and $t \to r$.
These directions have the following meaning:
\begin{equation} \label{eq-other-red}
  \text{if $r \to t$ then $\alpha_t$ belongs to the subrepresentation generated by $\alpha_r$.}
\end{equation}

We say $t \dle r$ if there is a sequence $t = s_1, s_2, \dots, s_n = r$ in $S$ such that $s_{i+1} \to s_i$ for all $i$.
Then $\dle$ is a pre-order on $S$.
The corresponding equivalence relation is denoted by $\dsim$, i.e., we say $r \dsim t$ if $r \dle t$ and $t \dle r$.

Immediately, we have the following corollary by the fact \eqref{eq-other-red}.

\begin{corollary}
Let $V$ be a generalized geometric representation as above.
\begin{enumerate}
    \item If all elements in $S$ are equivalent with respect to $\dsim$, then the action of $W$ on any proper subrepresentation of $V$ is trivial.
    \item Let $I \subsetneq S$ be an equivalent class with respect to $\dsim$, then $\bigoplus_{s \dle I} \mathbb{C} \alpha_s$ and $\bigoplus_{s \dle I, s \notin I} \mathbb{C} \alpha_s$ are subrepresentations of $V$.
    \item Let $V_I$ denote the quotient of representations 
        $$V_I := \bigl(\bigoplus_{s \dle I} \mathbb{C} \alpha_s \bigr) / \bigl(\bigoplus_{s \dle I, s \notin I} \mathbb{C} \alpha_s \bigr).$$
        Then either $V_I$ is irreducible, or $V_I$ admits a maximal subrepresentation $V_0$ on which $W$ acts trivially  such that $V_I / V_0$ is irreducible.
\end{enumerate}
\end{corollary}

\begin{example}
Suppose the Coxeter graph $G$ of the Coxeter group $(W,S)$ is as follows,
  
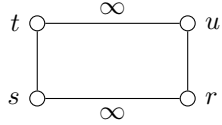
\begin{equation*}
  \begin{tikzpicture}
    \node [circle, draw, inner sep=2pt, label=left:$s$] (s) at (0,0) {};
    \node [circle, draw, inner sep=2pt, label=left:$t$] (t) at (0,1) {};
    \node [circle, draw, inner sep=2pt, label=right:$r$] (r) at (2,0) {};
    \node [circle, draw, inner sep=2pt, label=right:$u$] (u) at (2,1) {};
    \draw (s) -- (t) -- (u) -- (r) -- (s);
    \node (tu) at (1, 1.2) {$\infty$};
    \node (sr) at (1,-0.2) {$\infty$};
  \end{tikzpicture}
  \end{equation*}
  If we are given a GGR-datum such that $\delta_{tu} = \varrho_u^t$ and $\delta_{sr} = \varrho_r^s$, then $\mathbb{C} \langle \alpha_u, \alpha_r \rangle$ is a subrepresentation of the corresponding generalized geometric representation of $W$.
\end{example}

\subsection{Bilinear forms}
Different from Theorem \ref{prop-bil}\eqref{prop-bil-1}, it may happen that there is no nonzero $W$-invariant bilinear form on a generalized geometric representation $V$ even when $\chi (H_1 (\widetilde{G})) \subseteq \{\pm 1\}$.
For example, Let $S = \{s, t, r\}$, $m_{st} = \infty$, $m_{rt} = m_{sr} = 3$, and let $\delta_{st} = \varrho_s^t$.
Suppose $(-|-)$ is a $W$-invariant bilinear form on $V$.
By Lemma \ref{lem-2.11}\eqref{2.4.4}, it holds that $(\alpha_s| \alpha_s) = 0$.
By Lemma \ref{lem-2.7}, we have $ (\alpha_r| \alpha_r) = 0$ and then $(\alpha_t| \alpha_t) = 0$.
Therefore $(-|-)$ has to be zero.

If for any pair $r,t \in S$ with $r \ne t$ it holds $\delta_{rt} \neq \varrho_r^t$ and $\varrho_t^r$, then Theorem \ref{prop-bil}\eqref{prop-bil-1} still works.

\section{Reflection representations} \label{sec-refl-rep}

In this section, we investigate reflection representations without any linearly independence condition on the reflection vectors.

\subsection{Reflection representations and R-representations}

Let $(V, \rho)$ be a reflection representation of $(W,S)$.
By Lemma \ref{lem-refl},  the reflection vectors $\{\alpha_s \mid s \in S\}$ span a subrepresentation $U$, and $W$ acts on the quotient $V/U$ trivially.
Thus, there is no harm in assuming that
\begin{equation*}
  \text{$V$ is spanned by $\{\alpha_s \mid s \in S\}$.}
\end{equation*}

Recall that the definition of the R-representations requires that the vectors $\alpha_r$ and $\alpha_t$ are not proportional whenever $m_{rt} < \infty$.
Thus it remains to consider the possibility that $\alpha_r = \alpha_t$ (after rescaling) for some $r,t \in S$ with $m_{rt} < \infty$.
However, the following lemma is saying that in this case the action of $r$ and $t$ must be equal, i.e., $\rho(r) = \rho(t)$.

\begin{lemma} \label{lem-refl-6}
  Let $V$ be a vector space and $0 \ne \alpha \in V$.
  Suppose $\sigma, \tau : V \to V$ are two reflections such that $\sigma \cdot \alpha = \tau \cdot \alpha = - \alpha$.
  If there exists $m \in \mathbb{N}$ such that $(\sigma \tau)^m = \Id_V$, then $\sigma = \tau$.
\end{lemma}

\begin{proof}
  Let $H_\sigma$ and $H_\tau$ be the reflection hyperplanes of $\sigma$ and $\tau$ respectively.
  Suppose $H_\sigma \ne H_\tau$.
  Then there exists $v \in H_\tau \setminus H_\sigma$.
  Note that $V = H_\sigma \oplus \mathbb{C} \alpha$.
  Thus we can write $v = v_\sigma + a \alpha$ where $v_\sigma \in H_\sigma$ and $a \in \mathbb{C}^\times$.
  By direct computations, we have 
  $$(\sigma \tau) \cdot v = \sigma \cdot v = v_\sigma - a \alpha = v - 2a \alpha.$$
  Therefore, $(\sigma \tau)^m \cdot v = v - 2ma \alpha$ for any $m \in \mathbb{N}$, which contradicts $(\sigma \tau)^m = \Id_V$.
  Thus, $H_\sigma = H_\tau$ and then $\sigma  = \tau$.
\end{proof}

Now we return to the reflection representation $(V,\rho)$.
We define an equivalence relation $\papp$ on the set $S$ as follows.
We write $r \psim t$ if $\alpha_r$ and $\alpha_t$ are proportional and $m_{rt} <\infty$.
Clearly, $\psim$ is a reflexive relation.
Let $\papp$ be the equivalence relation generated by $\psim$, i.e., $r \papp t$ if $r \psim r_1 \psim \cdots \psim r_k \psim t$ for some $r_1, \dots, r_k \in S$.

If $r,t \in S$ such that $r \papp t$, then by definition $\alpha_r$ and $\alpha_t$ are proportional.
After rescaling, we can set $\alpha_r = \alpha_t$ whenever $r \papp t$.
Therefore, if $S = I_1 \sqcup \dots \sqcup I_k$ is the partition of $S$ into equivalence classes with respect to $\papp$, then we can denote by $\alpha_i$ the reflection vector of elements in $I_i$ ($i = 1, \dots, k$), and then $V = \sum_{1 \le i \le k} \mathbb{C} \alpha_i$.
Let $\{ s_1, \dots, s_k  \mid  s_i \in I_i \}$ be a set of representatives.

\begin{lemma} \label{lem-6-2}
  Let $(V,\rho)$, $\papp$, $\alpha_1, \dots, \alpha_k$, $s_1, \dots, s_k$ be as above.
  \begin{enumerate}
    \item \label{lem-6-2-1} If $r,t\in S$ such that $r \papp t$, then $\rho(r) = \rho(t)$.
    \item The image $\rho(W) \subseteq GL(V)$ is generated by $\rho(s_1), \dots, \rho(s_k)$ as a group.
    \item \label{lem-6-2-3} If $r \in I_i$ and $t \in I_j$ with $i \ne j$, and if $m_{rt} < \infty$, then $\alpha_i$ and $\alpha_j$ are linearly independent.
  \end{enumerate}
\end{lemma}

\begin{proof}
  Clear by Lemma \ref{lem-refl-6} and the definition of $\papp$.
\end{proof}

For two indices $i,j$ with $i \ne j$, we set
\begin{equation}   \label{eq-dij}
  d_{ij} := \gcd \{m_{rt} \mid r \in I_i, t \in I_j, m_{rt} < \infty\}.
\end{equation}
(Here ``gcd'' means the greatest common divisor.)
By convention, we set $\gcd \emptyset = \infty$.
It is clear that $d_{ij} = d_{ji}$.
Moreover, we have the following fact.

\begin{lemma} \label{lem-6-3} Let $i,j$ be two indices with $i \ne j$.
  \begin{enumerate}
    \item \label{lem-6-3-1} If $d_{ij} < \infty$, then $(\rho(s_i) \rho(s_j))^{d_{ij}} = \Id_V$.
    \item \label{lem-6-3-2} In particular, $d_{ij} > 1$.
  \end{enumerate}
\end{lemma}

\begin{proof}
  For any $r \in I_i$ and $t \in I_j$ with $m_{rt} < \infty$, by Lemma \ref{lem-6-2} we have $\rho(s_i) = \rho(r)$, $\rho(s_j) = \rho(t)$, and hence $(\rho(s_i) \rho(s_j))^{m_{rt}} = \Id_V$.
  Since $d_{ij}$ is the greatest common divisor of such $m_{rt}$'s, we have $(\rho(s_i) \rho(s_j))^{d_{ij}} = \Id_V$.
  
  If $d_{ij} = 1$, then $\rho(s_i) = \rho(s_j)$. 
  But $d_{ij} < \infty$ implies that $m_{rt} < \infty$ for some $r \in I_i$ and $t \in I_j$.
  By Lemma \ref{lem-6-2}\eqref{lem-6-2-3} this is impossible.
\end{proof}

Now we define a Coxeter group $(\overline{W}, \overline{S})$ as follows. 
Let $\overline{S} := \{\theta_1, \dots, \theta_k\}$, and $\overline{W}$ be the Coxeter group generated by $\overline{S}$ with Coxeter matrix $(d_{ij})_{1 \le i,j \le k}$ (we set $d_{ii} = 1$).
By Lemma \ref{lem-6-3} we have the following factorization of $\rho$.

\begin{theorem} \label{thm-refl}
  Notations as above.
  \begin{enumerate}
    \item \label{thm-refl-1} We have a surjective homomorphism $\pi_V : W \twoheadrightarrow \overline{W}$ such that $\pi_V (s) = \theta_i$ if $s \in I_i$.
    \item \label{thm-refl-2} The map $\theta_i \mapsto \rho(s_i)$ defines an R-representation $\overline{\rho}: \overline{W} \to GL(V)$.
    \item \label{thm-refl-3} The reflection representation $(V, \rho)$ of $W$ factors through $(\overline{W}, \overline{S})$ as follows,
        \begin{equation*}
           \xymatrix{W \ar[rr]^{\rho} \ar@{->>}[rd]_{\pi_V}& & GL(V) \\
             & \overline{W} \ar[ur]_{\overline{\rho}} &}
         \end{equation*}
  \end{enumerate}
\end{theorem}

\begin{proof}
  For \eqref{thm-refl-1}, it suffices to show $(\pi_V(r) \pi_V(t))^{m_{rt}} = e_{\overline{W}}$ whenever $r,t \in S$ with $m_{rt} < \infty$. 
  If $r = t$ or $r, t$ belong to the same $I_i$, then this is clear.
  Suppose $r \in I_i$ and $t \in I_j$ with $i \ne j$.
  Then we have 
  \begin{equation*}
    (\pi_V(r) \pi_V(t))^{m_{rt}} = (\theta_i \theta_j)^{m_{rt}} = (\theta_i \theta_j)^{d_{ij} \cdot \frac{m_{rt}}{d_{ij}}} = e_{\overline{W}}
  \end{equation*}
  as desired.
  
  By Lemma \ref{lem-6-3}\eqref{lem-6-3-1} and the fact that $\rho(s_i)$ is a reflection on $V$, the map $\theta_i \mapsto \rho(s_i)$ defines a reflection representation of $\overline{W}$ on $V$.
  If $d_{ij} < \infty$ for some $i,j$, then $m_{rt} < \infty$ for some $r \in I_i$ and $t \in I_j$.
  By Lemma \ref{lem-6-2}\eqref{lem-6-2-3}, $\alpha_i$ and $\alpha_j$ are linearly independent.
  This proves that $(V, \overline{\rho})$ is an R-representation of $\overline{W}$ (see Definition \ref{def-R}).
  
  The point \eqref{thm-refl-3} follows from Lemma \ref{lem-6-2}\eqref{lem-6-2-1}:
  \begin{equation*}
    \rho(s) = \rho(s_i) = \overline{\rho} (\theta_i) = \overline{\rho} \pi_V (s)
  \end{equation*}
  if $s \in I_i$.
\end{proof}

\begin{remark}
  The partition $S = I_1 \sqcup \cdots \sqcup I_k$, the group $(\overline{W}, \overline{S})$, and the R-representation $\overline{\rho}: \overline{W} \to GL(V)$ are uniquely determined by the isomorphism class of the reflection representation $\rho: W \to GL(V)$.
\end{remark}

In conclusion, a reflection representation of $W$ factors through an R-represen\-tation of a quotient group $\overline{W}$ of $W$.

\subsection{Admissible partitions}

In order to classify all the reflection representations of $(W, S)$, we need to determine which reflection vectors can be proportional.
This is equivalent to the question of what kind of partition $S = I_1 \sqcup \cdots \sqcup I_k$ will occur as in the last subsection. 
Lemma \ref{lem-6-3}\eqref{lem-6-3-2} indicates that not every partition of $S$ is permissible.
But in fact, the restriction in Lemma \ref{lem-6-3}\eqref{lem-6-3-2} is almost sufficient.

\begin{definition}
  Let $S = I_1 \sqcup \cdots \sqcup I_k$ be a partition of the set $S$.
  As in \eqref{eq-dij}, for two indices $i,j$ with $1 \le i \ne j \le k$ we set
  \begin{equation*}
    d_{ij} := \gcd \{m_{rt} \mid r \in I_i, t \in I_j, m_{rt} < \infty\}.
  \end{equation*}
  (Recall that we set $\gcd \emptyset = \infty$.)
  The partition is called an \emph{admissible partition} of $S$ 
  if the following two conditions are satisfied:
  \begin{enumerate}
    \item  $d_{ij} > 1$ for any indices $i,j$ with $i \ne j$;
    \item for each index $i$, $I_i$ can not be written as a disjoint union 
      $I_i = J \sqcup J^\prime$ such that $m_{rt} = \infty$ 
      for any $r \in J$ and $t \in J^\prime$.
  \end{enumerate}
\end{definition}

Let $S = I_1 \sqcup \cdots \sqcup I_k$ be an admissible partition.
We set $d_{ii} = 1$ for any $i$.
Then $(d_{ij})_{1 \le i,j \le k}$ is a Coxeter matrix.
We define a Coxeter group $(\overline{W}, \overline{S})$ with Coxeter matrix $(d_{ij})_{1 \le i,j \le k}$ and $\overline{S} = \{\theta_1, \dots, \theta_k\}$ as in the last subsection.
Then by the same reason of Theorem \ref{thm-refl}\eqref{thm-refl-1} we have a surjective homomorphism of groups $\pi : W \twoheadrightarrow \overline{W}$ such that $\pi (s) = \theta_i$ if $s \in I_i$.
Clearly, any R-representation of $\overline{W}$ pulls back to a reflection representation of $W$ along $\pi$.
Therefore we have the following classification of reflection representations.

\begin{theorem}
  The isomorphism classes of reflection representations (spanned by reflection vectors) of $W$ one-to-one correspond to the data 
  $\{(S = I_1 \sqcup \cdots \sqcup I_k, (V, \overline{\rho}))\}$
  where
  \begin{enumerate}
    \item $S = I_1 \sqcup \cdots \sqcup I_k$ is an admissible partition of $S$;
    \item $(V, \overline{\rho})$ is an (isomorphism class of) R-representation of $\overline{W}$, where $\overline{W}$ is a quotient of $W$ defined by the partition as in the last subsection.
  \end{enumerate}
\end{theorem}

\begin{example}
  The trivial partition of $S$ (i.e., $k=1$) corresponds to the sign representation of $W$.
\end{example}

The condition $d_{ij} > 1$ is in fact a strict restriction, as we shall see in the following examples. 
We say a partition $S = I_1 \sqcup \cdots \sqcup I_k$ is \emph{discrete} if each part $I_i$ consists of a single element.

\begin{example} \label{eg-610}
  Let $W$ be the symmetric group $\mathfrak{S}_{n+1}$, with Coxeter graph
  
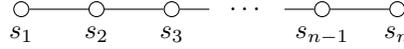
\begin{equation*}
    \begin{tikzpicture}
    \node [circle, draw, inner sep=2pt, label=below:$s_1$] (s0) at (-3,0) {};
    \node [circle, draw, inner sep=2pt, label=below:$s_2$] (s1) at (-2,0) {};
    \node [circle, draw, inner sep=2pt, label=below:$s_3$] (s2) at (-1,0) {};
    \node [circle, draw, inner sep=2pt, label=below:$s_{n-1}$] (sn-1) at (1,0) {};
    \node [circle, draw, inner sep=2pt, label=below:$s_{n}$] (sn) at (2,0) {};
    \draw (sn-1) -- (sn);
    \draw (s0) -- (s1) -- (s2);
    \draw (s2) -- (-0.5,0);
    \draw (sn-1) -- (0.5,0);
    \node (d) at (0,0) {$\dots$};
    \end{tikzpicture}
  \end{equation*}
  
  If $n = 3$, the partition $S = \{s_1, s_3\} \sqcup \{s_2\}$ is the only admissible partition which is non-trivial and non-discrete.
  
  Suppose $n \ge 4$, $S = I_1 \sqcup \cdots \sqcup I_k$ is an admissible partition and $s_i, s_j \in I_1$ with $1 \le i < j \le n$.
  If $s \in S$ is a vertex in the Coxeter graph which is adjacent to one of $s_i$ and $s_j$ while non-adjacent to the other (e.g., $s= s_{i-1}$ if $i > 1$, or $s= s_{i+1}$ if $j-i>2$), then $s$ must belong to $I_1$ as well because $m_{ss_i}$ and $m_{ss_j}$ are co-prime.
  In this way one easily deduces that the partition of $S$ must be trivial or discrete.
  In other words, a reflection representation of $\mathfrak{S}_{n+1}$ ($n \ge 4$) is either the sign representation or an R-representation (in fact, the only R-representation of $\mathfrak{S}_{n+1}$ is the geometric representation $\Vg$).
\end{example}

\begin{example}
  Consider the Coxeter group of type $\widetilde{\mathsf{A}}_n$ (see Example \ref{eg-affine-A}).
  
  If $n = 2$, then $S = \{s_0\} \sqcup \{s_1, s_2\}$ is an admissible partition.
  
  If $n = 3$, then $S = \{s_0, s_2\} \sqcup \{s_1, s_3\}$ is the only non-trivial and non-discrete admissible partition.
  
  If $n \ge 4$, then, similar to Example \ref{eg-610}, any admissible partition of $S$ is either trivial or discrete.
\end{example}

\bibliographystyle{amsplain}
\bibliography{refl-rep}

\providecommand{\bysame}{\leavevmode\hbox to3em{\hrulefill}\thinspace}
\providecommand{\MR}{\relax\ifhmode\unskip\space\fi MR }
\providecommand{\MRhref}[2]{%
  \href{http://www.ams.org/mathscinet-getitem?mr=#1}{#2}
}
\providecommand{\href}[2]{#2}
\begin{thebibliography}{10}

\bibitem{BB05}
Anders Bj\"{o}rner and Francesco Brenti, \emph{Combinatorics of {C}oxeter
  {G}roups}, Graduate Texts in Mathematics, vol. 231, Springer, New York, 2005.

\bibitem{Bourbaki-Lie456}
Nicolas Bourbaki, \emph{{Lie Groups and Lie Algebras. Chapters 4--6}}, Elements
  of Mathematics, Springer-Verlag, Berlin, 2002, translated from the 1968
  French original by Andrew Pressley.

\bibitem{BCNS15}
Vadim Bugaenko, Yonah Cherniavsky, Tatiana Nagnibeda, and Robert Shwartz,
  \emph{Weighted {C}oxeter graphs and generalized geometric representations of
  {C}oxeter groups}, Discrete Appl. Math. \textbf{192} (2015), 17--27.

\bibitem{Coxeter34}
H.~S.~M. Coxeter, \emph{Discrete groups generated by reflections}, Ann. of
  Math. (2) \textbf{35} (1934), no.~3, 588--621.

\bibitem{Coxeter35}
\bysame, \emph{The complete enumeration of finite groups of the form
  {$R_i^2=(R_iR_j)^{k_{ij}}=1$}}, J. London Math. Soc. \textbf{10} (1935),
  no.~1, 21--25.

\bibitem{DPWX22}
Ivan Dimitrov, Charles Paquette, David Wehlau, and Tianyuan Xu,
  \emph{Subregular {$J$}-rings of {C}oxeter systems via quiver path algebras},
  J. Algebra \textbf{612} (2022), 526--576.

\bibitem{Donnelly11}
Robert~G. Donnelly, \emph{Root systems for asymmetric geometric representations
  of {C}oxeter groups}, Comm. Algebra \textbf{39} (2011), no.~4, 1298--1314.

\bibitem{Hee91}
Jean-Yves H\'{e}e, \emph{Syst\`eme de racines sur un anneau commutatif
  totalement ordonn\'{e}}, Geom. Dedicata \textbf{37} (1991), no.~1, 65--102.

\bibitem{Hu23}
Hongsheng Hu, \emph{Representations of {C}oxeter groups of {L}usztig's
  $\boldsymbol{a}$-function value 1}, preprint, arXiv:2309.00593, 2023.

\bibitem{Humphreys90}
James~E. Humphreys, \emph{{Reflection Groups and Coxeter Groups}}, Cambridge
  Studies in Advanced Mathematics, vol.~29, Cambridge University Press,
  Cambridge, 1990.

\bibitem{KL79}
David Kazhdan and George Lusztig, \emph{Representations of {C}oxeter groups and
  {H}ecke algebras}, Invent. Math. \textbf{53} (1979), no.~2, 165--184.

\bibitem{Krammer09}
Daan Krammer, \emph{The conjugacy problem for {C}oxeter groups}, Groups Geom.
  Dyn. \textbf{3} (2009), no.~1, 71--171.

\bibitem{Lusztig83-intrep}
George Lusztig, \emph{Some examples of square integrable representations of
  semisimple {$p$}-adic groups}, Trans. Amer. Math. Soc. \textbf{277} (1983),
  no.~2, 623--653.

\bibitem{Munkres84}
James~R. Munkres, \emph{Elements of {A}lgebraic {T}opology}, Addison-Wesley
  Publishing Company, Menlo Park, CA, 1984.

\bibitem{Serre77}
Jean-Pierre Serre, \emph{{Linear Representations of Finite Groups}}, Graduate
  Texts in Mathematics, vol.~42, Springer-Verlag, New York-Heidelberg, 1977,
  translated from the second French edition by Leonard L. Scott.

\bibitem{Sunada13}
Toshikazu Sunada, \emph{{Topological Crystallography. With a View Towards
  Discrete Geometric Analysis}}, Surveys and Tutorials in the Applied
  Mathematical Sciences, vol.~6, Springer, Tokyo, 2013.

\bibitem{Vinberg71}
\`Ernest~Borisovich Vinberg, \emph{Discrete linear groups generated by
  reflections}, Math. USSR, Izv. \textbf{5} (1971), 1083--1119.

\end{thebibliography}

\end{document}